\documentclass[11pt]{article}

\usepackage {amsmath}
\usepackage {hyperref}
\usepackage {amsfonts}
\usepackage {amsthm}
\usepackage[parfill]{parskip}
\usepackage {amssymb}
\usepackage {fullpage}
\usepackage {mathrsfs}
\usepackage {amscd}
\usepackage[all,cmtip]{xy}
\DeclareMathAlphabet{\mathpzc}{OT1}{pzc}{m}{it}

\title{Crossed $S$-matrices and Character Sheaves on Unipotent Groups}
\author{Tanmay Deshpande}
\date{}

\newtheorem {thm} {Theorem} [section]
\newtheorem {prop} [thm] {Proposition}
\newtheorem {conj} [thm] {Conjecture}
\newtheorem {lem} [thm] {Lemma}
\newtheorem {cor} [thm] {Corollary}

\theoremstyle{definition}
\newtheorem {defn} [thm] {Definition}

\newtheorem {prob} [thm]  {Problem}

\newtheorem {rk} [thm]  {Remark}
\newtheorem {ex} [thm] {Example}

\newcommand{\beq}{\begin{equation}}
\newcommand{\eeq}{\end{equation}}
\newcommand{\bthm}{\begin {thm}}
\newcommand{\ethm}{\end {thm}}
\newcommand{\bprop}{\begin {prop}}
\newcommand{\eprop}{\end {prop}}
\newcommand{\bprob}{\begin {prob}}
\newcommand{\eprob}{\end {prob}}
\newcommand{\bcor}{\begin {cor}}
\newcommand{\ecor}{\end {cor}}
\newcommand{\blem}{\begin{lem}}
\newcommand{\elem}{\end{lem}}
\newcommand{\bdefn}{\begin{defn}}
\newcommand{\edefn}{\end{defn}}
\newcommand{\brk}{\begin{rk}}
\newcommand{\erk}{\end{rk}}

\newcommand{\xto}{\xrightarrow}
\newcommand{\onto}{\twoheadrightarrow}

\newcommand{\av}{{\hbox{av}}}
\newcommand{\tMtg}{\tM_{U\tg,e}}

\newcommand{\tMeg}{\tM_{\t{G},e}}

\newcommand{\tG}{\t{G}}
\newcommand{\tGam}{\t{\Gamma}}
\newcommand{\Mtge}{\M_{\t{G},e}}

\newcommand{\Sh}{\h{Sh}}
\renewcommand {\bar} {\overline}
\renewcommand {\mapsto} {\longmapsto}
\renewcommand {\to} {\rar{}}
\newcommand{\bpf}{\begin{proof}}
\newcommand{\epf}{\end{proof}}
\newcommand{\bex}{\begin{ex}}
\newcommand{\eex}{\end{ex}}
\newcommand{\rar}[1]{\stackrel{#1}{\longrightarrow}}
\newcommand{\f}{\mathbb}
\newcommand{\bZ}{\mathbb{Z}}

\newcommand{\tensor}{\otimes}
\newcommand{\h}{\operatorname}
\newcommand{\FPdim}{\operatorname{FPdim}}
\newcommand{\ind}{\operatorname{ind}}
\newcommand{\Ind}{\operatorname{Ind}}
\newcommand{\Res}{\operatorname{Res}}

\newcommand{\indg}{\h{ind}_{G'}^G}
\newcommand{\Resg}{\h{Res}_{G'}^G}

\renewcommand{\L}{\mathcal{L}}

\newcommand{\cM}{\C}

\newcommand{\Fp} {\mathbb{F}_p}
\newcommand{\Fq} {\mathbb{F}_q}
\newcommand{\tK} {\widetilde{K}}

\newcommand{\Fqcl} {\overline{\mathbb{F}}_q}
\newcommand{\Q} {\mathbb{Q}}

\newcommand{\F} {\mathbb{F}}
\newcommand{\un} {\mathbf{1}}

\newcommand{\normal} {\triangleleft}
\renewcommand{\t}{\widetilde}

\renewcommand{\k} {\mathtt{k}}

\newcommand{\M} {\mathcal{M}}

\newcommand{\D} {\mathscr{D}}
\newcommand{\C} {\mathcal{C}}

\newcommand{\Hom} {\h{Hom}}
\newcommand{\RHom} {\h{RHom}}
\renewcommand{\Vec}{\h{Vec}}
\newcommand{\Vect}{\h{Vec}}
\newcommand{\Aut}{\h{Aut}}

\newcommand{\Stab}{\h{Stab}}

\newcommand{\Qlcl} {\overline{\mathbb{Q}}_l}

\newcommand{\G} {\mathbb{G}}

\newcommand{\Gtq} {G^t_0(\Fq)}
\newcommand{\Ggq} {G^g_0(\Fq)}

\newcommand{\noin}{\noindent}

\newcommand{\tg} {{\tilde{g}}}

\newcommand{\g} {{\gamma}}
\renewcommand{\l} {{\lambda}}

\newcommand{\id}{\hbox{id}}

\newcommand{\Meg}{\M_{G,e}}
\newcommand{\tM}{\widetilde{\M}}

\newcommand{\bit}{\begin{itemize}}
\newcommand{\eit}{\end{itemize}}

\newcommand{\Rep}{\h{Rep}}
\newcommand{\Irrep}{\h{Irrep}}

\newcommand{\DG}{\D_G(G)}
\newcommand{\FunG}{\Fun([G_0](\Fq))}
\newcommand{\DGp}{\D_{G'}(G')}

\newcommand{\DGn}{\D_{G_0}(G_0)}

\newcommand{\can}{\mathbb{K}}

\renewcommand{\r}{\mathfrak{r}}

\renewcommand{\O}{\mathcal{O}}
\newcommand{\E}{\mathcal{E}}

\newcommand{\Z}{\mathbb{Z}}
\renewcommand{\b}{\beta}
\newcommand{\tr}{\h{tr}}
\newcommand{\Fun}{\h{Fun}}

\newcommand{\bconj}{\begin{conj}}
\newcommand{\econj}{\end{conj}}
\newcommand{\DGt}{\D_G(\t{G})}
\newcommand{\DGtp}{\D_{G'}(\t{G}')}

\begin{document}

\maketitle

\begin{abstract}
\noin Let $\k$ be the algebraic closure of a finite field ${\F}_{q}$ of characteristic $p$. Let $G$ be a connected unipotent group over $\k$ equipped with an $\Fq$-structure given by a Frobenius map $F:G\to G$. We will denote the corresponding algebraic group defined over $\Fq$ by $G_0$. Character sheaves on $G$ are certain special objects in the triangulated braided monoidal category $\D_G(G)$ of bounded conjugation equivariant $\Qlcl$-complexes (where $l\neq p$ is a prime number) on $G$. Boyarchenko has proved that the ``trace of Frobenius'' functions associated with $F$-stable character sheaves on $G$ form an orthonormal basis of the space of class functions on $G_0(\Fq)$ and that the matrix relating this basis to the basis formed by the irreducible characters of $G_0(\Fq)$ is block diagonal with ``small'' blocks. In particular, there is a partition of the set of character sheaves as well as a partition of the set of irreducible characters of $G_0(\Fq)$ into ``small'' families know as $\f{L}$-packets.  In this paper we describe these block matrices relating character sheaves and irreducible characters corresponding to each $\f{L}$-packet. We prove that these matrices can be described as certain ``crossed $S$-matrices'' associated with each $\f{L}$-packet. We will also derive a formula for the dimensions of the irreducible representations of $G_0(\Fq)$ in terms of certain modular categorical data associated with the corresponding $\f{L}$-packet.  In fact we will formulate and prove more general results which hold for possibly disconnected groups $G$ such that $G^\circ$ is unipotent. To prove our results, we will establish a formula (which holds for any algebraic group $G$) which expresses the inner product of the ``trace of Frobenius'' function of any $F$-stable object of $\DG$ with any character of $G_0(\Fq)$ (or of any of its pure inner forms) in terms of certain categorical operations.
\end{abstract}

\tableofcontents

\section{Introduction}

Let $\k=\bar{\f{F}}_q$, where $q$ is a power of a prime $p.$ Let $G$ be an algebraic group (or perfect quasi-algebraic group, see \S\ref{nac}) over $\k$ equipped with an $\Fq$ structure given by a Frobenius map $F:G\to G$. Let us denote the corresponding algebraic group defined over $\Fq$ by $G_0$. Let us fix a prime number $l\neq p.$ One of the main goals of the theory of character sheaves is to study the $\Qlcl$-valued irreducible characters of the groups $G_0(\Fq)$ using the machinery of $\Qlcl$-sheaves and complexes on $G$.

In the series of papers \cite{L}, Lusztig developed a theory of character sheaves on reductive groups and used it to study the characters of finite groups of Lie type. Inspired by Lusztig's work, Boyarchenko and Drinfeld (in \cite{BD1}, \cite{BD2}, \cite{Bo1}, \cite{Bo2}) developed a theory of character sheaves on unipotent groups and related it to the character theory of finite unipotent groups. In both these cases, character sheaves are isomorphism classes of certain objects in the $\Qlcl$-linear triangulated braided monoidal category $\D_G(G)$ of conjugation equivariant $\Qlcl$-complexes on $G$. We have an action of Frobenius on the set of character sheaves, or in other words pullback by Frobenius of a character sheaf is also a character sheaf.

Let us suppose that $G$ is a {\it connected unipotent} group. In \cite{Bo2}, Boyarchenko proves that there is a bijection between the set of irreducible characters of $G_0(\Fq)$ and the set of isomorphism classes of character sheaves fixed by Frobenius. More precisely he shows that the $\Qlcl$-valued ``trace of Frobenius'' functions on $G_0(\Fq)$ associated with $F$-stable character sheaves form an orthonormal basis for the space of $\Qlcl$-valued class functions on $G_0(\Fq)$ and that the matrix relating this basis to the basis formed by the irreducible characters of $G_0(\Fq)$ is block diagonal with ``small'' blocks. A similar result is proved by Lusztig in \cite{L} for reductive groups.

Now let $G$ be (possibly disconnected) such that its neutral connected component $G^\circ$ is unipotent. (The notion of character sheaves on $G$ is well defined in this case.) In this case, the number of $F$-stable character sheaves on $G$ may be  strictly larger than the number of irreducible characters of the group $G_0(\Fq)$. As observed in \cite{Bo2}, in this situation it is more natural to consider all pure inner forms (see \S\ref{pif}) of $G_0$. One reason for this is that if $C\in \D_G(G)$ is such that we have an isomorphism $\psi:F^*(C)\xto{\cong}C$, then we can define ``trace of Frobenius'' function $t_{C,\psi}^g:G^g_0(\Fq)\to \Qlcl$ for the pure inner form $G^g_0(\Fq)$ corresponding to each $g\in G$ (see \S\ref{sfcorr}). The above discrepancy between the number of characters and character sheaves disappears once we consider the irreducible characters of the pure inner forms as well. Once again the ``trace of Frobenius functions'' of $F$-stable character sheaves satisfy similar orthogonality relations and the matrix relating them to the irreducible characters of all pure inner forms is again block diagonal as in the connected case. (See \cite{Bo2} for more.)

The main goal of this paper is to describe these block matrices in the (block diagonal) change of basis matrix. It is shown in \cite{Bo2} that these blocks are labeled by $F$-stable minimal idempotents $e\in \DG$. Now as proved in \cite{BD2}, we have the modular category $\M_{G,e}\subset e\DG$ associated with the minimal idempotent $e$. Since $e$ is $F$-stable, $F^*$ induces an autoequivalence of the modular category $\M_{G,e}$. We prove that the block matrix labeled by $e$ is the ``crossed $S$-matrix'' (see \S\ref{bcgandcsm}) associated with the modular category $\Meg$ and its autoequivalence $F^*$.

In the first part of the paper, we work with a general algebraic group $G$ and derive an inner product formula (Theorem \ref{ipf}) for the inner product of the ``trace of Frobenius'' function of a $F$-stable object in $\DG$ and the character of a representation of an inner form $G^g_0(\Fq)$. We then use this formula to describe the block matrices for groups having a unipotent neutral connected component in \S\ref{mrc}. We believe that the inner product formula (Theorem \ref{ipf}) would be useful to describe the relationship between character sheaves and irreducible characters of general and in particular reductive algebraic groups as well. 

We will formulate all the main results of this paper in \S\ref{summary} and prove them in subsequent sections. In \S\ref{s:examples}, we will study two instructive examples in detail. Here we will also encounter some interesting examples of crossed $S$-matrices and nontrivial $\f{L}$-packets. In \S\ref{grows} we prove that the algebra of class functions on all the pure inner forms $G^g_0(\Fq)$ is isomorphic to a certain Grothendieck ring of $F$-stable character sheaves on a neutrally unipotent group $G$.

\section*{Acknowledgments}
I am grateful to V. Drinfeld for introducing me to the theory of character sheaves and for many useful discussions and correspondence. I would also like to thank M. Boyarchenko and V. Ostrik for useful correspondence. This work was supported by World Premier Institute Research Center Initiative (WPI), MEXT, Japan.

\section{Preliminaries and main results}\label{summary}

In this section we describe the main definitions and constructions required in this paper, state the main results and describe the organization of the text.

\subsection{Notation and conventions}\label{nac}
We fix a prime number $p$ and $q$, a power of $p$. The field $\k$ will always denote an algebraically closed field of characteristic $p$. Typically we will assume that $\k=\bar{\F}_q$. All algebraic groups and schemes are assumed to be over $\k$ unless explicitly stated otherwise. We also fix a prime number $l$ different from $p$. For a scheme $X$ over $\k$, we set $\D(X):=D^b_c(X,\Qlcl)$, the $\Qlcl$-linear triangulated category of bounded constructible $\Qlcl$-complexes on $X$. If an algebraic group $G$ acts on $X$, we let $\D_G(X)$ denote the equivariant derived category. It is defined as the category of $\Qlcl$-complexes on the quotient stack $G\backslash X$. If $G$ is such that $G^\circ$ is unipotent, then objects of $\D_G(X)$ may be thought of as pairs $(M,\phi)$, where $M\in \D(X)$ and $\phi$ is a $G$-equivariance structure on $M$. We refer to \cite{BD2} for details. We will omit the symbols $L,R$ when talking about Grothendieck's six functors and all such functors should be interpreted in the derived sense. We denote the Verdier duality functor by $\f{D}$.

If $G$ is an algebraic group, then $\D(G)$ has the structure of a monoidal category under convolution with compact supports.  The  unit object of this monoidal category is the delta sheaf $\delta_1$ supported at the identity $1\in G$. The conjugation equivariant category $\D_G(G)$ has the structure of a braided monoidal category. We refer to \cite{BD2} for complete definitions. Let $\iota:G\to G$
 be the inversion map. The functor $\f{D}^-:=\iota^*\f{D}=\f{D}\iota^*$ defines a duality in the categories $\D(G)$ and $\DG$ which is weaker than rigidity. This weak duality makes $\D(G)$ and $\DG$ $\r$-categories (see \cite[\S A]{BD2}). The natural identification $(\f{D}^-)^2\cong Id$ defines a structure of a ribbon $\mathfrak{r}$-category on $\DG$. We refer to \cite[\S A]{BD2} for a detailed exposition.

We will often work in the setting of perfect schemes and perfect group schemes over $\k$. We refer to \cite[\S1.9]{BD2} for more details on the notions of perfect groups, perfect schemes and the perfectization functor. We often abuse notation and use the same symbol to donate a scheme (or algebraic group) and its perfectization.
This is not likely to cause any confusion since the categories $\D_G(X)$ and the groups $G(\k)$, $G_0(\Fq)$  do not change after passing to the perfectization. There are many advantages of passing to the perfectization. For example, after passing to the perfectization the Frobenius map $F$ becomes an isomorphism. Also perfectness is needed in the very useful notion of \emph{Serre duality} of connected unipotent groups. Hence from now on by schemes we in fact mean perfect quasi-algebraic schemes (i.e. perfectizations of schemes of finite type) and by algebraic groups we mean perfect quasi-algebraic groups (i.e. perfectizations of algebraic groups) even if we do not mention this explicitly.

\subsection{Sheaf-function correspondence}
Let $X$ be a scheme over $\k$ equipped with an $\Fq$-structure given by a Frobenius $F:X\to X$. Let $X_0$ be the corresponding scheme defined over $\Fq$. Let $\D^{Weil}(X_0)$ be the category consisting of pairs $(M,\psi)$ such that $M\in \D(X)$ and $\psi:F^*(M)\xto\cong M$ is an isomorphism in $\D(X)$. We call objects of $\D^{Weil}(X_0)$ as Weil complexes on $X$. We have the natural functor $\D(X_0)\to \D^{Weil}(X_0)$. If $(M,\psi)\in \D^{Weil}(X_0)$, then we have the stalk maps $\psi(x):M_{F(x)}\to M_x$ between complexes of $\Qlcl$-vector spaces. Then we define the trace of Frobenius function $t_{M,\psi}:X_0(\Fq)\to \Qlcl$ by $x\mapsto tr(\psi(x))$.

If $X$, $Y$ are schemes over $\k$ with an $\Fq$-structure and if $f:X\to Y$ is an $\Fq$-morphism, then we have the induced six functors $f_*,f_!,f^*,f^!,\otimes, \mathpzc{Hom}$ with all the standard adjunctions in the context of Weil complexes as well. By a slight abuse of notation we also let $f$ denote the induced map of finite sets $f:X_0(\Fq)\to Y_0(\Fq)$. In this context, $f^*$ also denotes pullback of functions on $Y_0(\Fq)$ along $f$, and $f_!$ also means summing up a function on $X_0(\Fq)$ along the fibers of $f$. The next lemma says that pullbacks, pushforwards with compact support and tensor product are compatible with the sheaf-function correspondence.
\blem (See \cite[Lem. 4.4]{Bo1}.) Let $f:X\to Y$ be as in the previous paragraph. \\
(i) If $N\in \D^{Weil}(Y_0)$, then $t_{f^*N}=f^*(t_N):=t_N\circ f$.\\
(ii) If $M,K\in \D^{Weil}(X_0)$, then $t_{M\otimes K}=t_M\cdot t_K.$\\
(iii) Assume that $f$ is also separated. If $M\in \D^{Weil}(X_0)$, then $t_{f_!(M)}=f_!(t_M)$.
\elem

An immediate consequence is the following:
\blem\label{compwithconv}
Let $G$ be an algebraic group with an $\Fq$-structure. Then if $M,N\in \D^{Weil}(G_0)$, then $t_{M\ast N}=t_M\ast t_N$, where the right hand side is the convolution of functions on the finite group $G_0(\Fq)$.
\elem

\subsection{Braided crossed categories and crossed $S$-matrices}\label{bcgandcsm}
Let $\Gamma$ be an abstract group. We briefly recall the notion of a braided $\Gamma$-crossed category.  We refer to \cite[\S 4.4.3]{DGNO} for a precise definition and properties of braided crossed categories and related concepts. A braided $\Gamma$-crossed category $\D$ is an additive monoidal category $(\D,\otimes,\un)$ equipped with the following structure:
\bit
\item A monoidal grading $\D=\bigoplus\limits_{a\in\Gamma}\C_a$. 
\item A monoidal action of $\Gamma$ on $\D$ such that $a(\C_b)\subset \C_{aba^{-1}}$ for each $a,b\in\Gamma$.
\item For $a\in \Gamma, M\in\C_a$ and $C\in \D$ isomorphisms (called crossed braiding isomorphisms)
$$\beta_{M,C}:M\otimes C\stackrel{\cong}{\rightarrow} a(C)\otimes M$$ functorial in $M,C$ and satisfying certain compatibility conditions which we do not explicitly recall here. These conditions in particular imply that $\C_1$ is a braided monoidal category and that the induced action of $\Gamma$ on $\C_1$ is a braided action.

\eit

We say that such a $\D$ is faithfully graded if $\C_a$ is nonzero for each $a\in \Gamma$. In this paper we will only encounter or consider faithfully graded categories.

Let $\D$ be such a category and let $a\in \Gamma.$ Let $M\in \C_a$ and $C\in \C_1$ be such that $a(C)\cong C$. Let us fix an isomorphism $\psi:a^{-1}(C)\xto{\cong}C$. We have the following composition
\beq\label{compisom}
C\otimes M\xto{\psi^{-1}\otimes \id_M}a^{-1}(C)\otimes M\xto{\beta_{a^{-1}(C),M}}M\otimes a^{-1}(C)\xto{\beta_{M,a^{-1}(C)}} C\otimes M \mbox{ in $\C_a$}.
\eeq 
Let $\g_{C,\psi,M}\in \Aut_{\C_a}(C\otimes M)
$ denote the inverse of the composition above.

\subsubsection{Module categories and crossed $S$-matrices}\label{modcatandcsm}
Let $\C$ be a ($\Qlcl$-linear) modular category and let $\M$ be an invertible $\C$-module category. Let $\O_\C$ and $\O_\M$ denote the sets of isomorphism classes of simple objects in $\C$ and $\M$ respectively. Let $\M$ be equipped with a $\C$-module trace $\tr_{\M}$. This means that we can define traces $\tr_\M(f)\in \Qlcl$ of endomorphisms $f$ in $\M$ and these are compatible with the spherical structure on $\C$. (We refer to \cite{S} for details.) In particular, using the module trace we can define categorical dimensions of objects of $\M$, namely if $M\in\M$, then $\dim(M):=\tr_{\M}(\id_M)$. The compatibility with the spherical structure on $\C$ implies that  $\dim(C\otimes M)=\dim(C)\cdot\dim(M)$ for each $C\in \C, M\in \M$. We impose an addition condition on the trace, namely we assume that the trace is normalized in such a way that 
\beq\label{normtr}
\sum\limits_{M\in \O_\M}\dim(M)^2=\dim(\C).
\eeq

Since $\M$ is an invertible $\C$-module category that admits a module trace we obtain using \cite[Thm. 5.2]{ENO} the corresponding autoequivalence $a:\C\to \C$ of modular categories which is unique up to natural equivalence of functors. Hence this induces a well defined permutation (which we also call $a$) of the set $\O_\C$. For $C\in \O_\C^a$, we choose an isomorphism $\psi:a^{-1}(C)\xto{\cong}C$. Then by definition of the autoequivalence $a$ we have the following composition in $\M$ analogous to (\ref{compisom}):
$$C\otimes M\xto{\cong}a^{-1}(C)\otimes M\xto{\cong}M\otimes a^{-1}(C)\xto{\cong} C\otimes M \mbox{ for $C\in \C, M\in \M$.}$$

We let $\g_{C,\psi,M}$ denote the inverse of this composition. 

\brk\label{crbrandmod}
Using \cite{ENO} we see that if $\D$ is a {\it spherical} (faithfully graded) braided $\Gamma$-crossed category whose trivial component $\C_1$ is a modular category then for each $a\in \Gamma$, the component $\C_a$ is an invertible $\C_1$-module category (with module trace induced by the spherical structure on $\D$) which satisfies the requirements of the first paragraph of this subsection. In particular, the $\C_1$-module trace on $\C_a$ obtained from the spherical structure of $\D$ is normalized in the sense of (\ref{normtr}). We note that in this situation, the automorphism $\g_{C,\psi,M}$ defined above matches with the one defined by (\ref{compisom}).
\erk

\bdefn\label{Smatrix}
Let $\C, \M$ and $a:\C\to \C$ be as before. For each $C\in\O_\C^a$ let us choose isomorphisms $\psi_C:a^{-1}(C)\to C$. We define the crossed $S$-matrix associated with the modular category $\C$ and the $\C$-module category $\M$ to be the matrix $S^{\C,\M}=\left(S^{\C,\M}_{C,M}\right)_{C\in \O_\C^a, M\in \O_\M}$ whose entries are defined by 
$$S^{\C,\M}_{C,M}:=\tr_\M(\g_{C,\psi_C,M}) \mbox{ for each $C\in \O_\C^a, M\in \O_\M$.}$$
\edefn

Our goal in this paper is to show that the blocks in the block diagonal matrix relating characters and character sheaves are precisely such crossed $S$-matrices for suitable choices of the categories $\C$ and $\M$ and to describe these categories corresponding to each block. 

\brk
The crossed $S$-matrix is well defined only up to a rescaling of the rows since the definition involves a choice of the isomorphisms $\psi_C$.
\erk

\brk
Let us consider $\C$ as an invertible $\C$-module category equipped with the trace coming from the spherical structure. We choose the autoequivalence $a=\id_{\C}$ and for each $C\in \O_\C$, we choose $\psi_C=\id_C$. With these choices $S^{\C,\C}$ equals the usual $S$-matrix of the modular category $\C$.
\erk

\subsection{Arbitrary algebraic groups}
We first discuss some generalities for arbitrary (perfect quasi-) algebraic groups $G$ over $\k$ and state our first main result (Theorem \ref{ipf}) which holds for arbitrary algebraic groups $G$ over $\k$ equipped with an $\Fq$-structure. In this subsection we assume that $G$ is any algebraic group over $\k$ and $F:G\to G$ is a Frobenius automorphism which equips $G$ with an $\Fq$-structure. We begin by recalling the notion of pure inner forms and explain their natural role in the theory of character sheaves.

\subsubsection{The $F$-conjugation action and pure inner forms}\label{pif}
We define the $F$-conjugation action of $G$ on itself by $g:h\mapsto ghF(g^{-1})$. Let us denote the set of $F$-conjugacy classes in $G$ by $H^1(F,G)$. By Lang's theorem, we have a natural bijection $H^1(F,G)\xto{\cong}H^1(F,\pi_0(G))$ and $H^1(F,G)$ is a finite set.

For each $g\in G$, the map $ad(g)\circ F:G\to G$ is a Frobenius map which defines a new $\Fq$-structure on $G$ and we denote the corresponding group defined over $\Fq$ by $G_0^g$ and call it the pure inner form of $G_0$ determined by $g\in G$. The isomorphism class of $G^g_0$ only depends on the $F$-conjugacy class of $g$ since we have a commutative diagram
$$\xymatrixcolsep{7pc}\xymatrix{
G\ar[r]^-{ad(g)\circ F}\ar[d]_{ad(h)} & G\ar[d]^{ad(h)}\\
G\ar[r]^-{ad(hgF(h^{-1}))\circ F} & G.\\
}$$

 If $I\subset G$ is a set of representatives of $F$-conjugacy classes in $G$, we consider the finite set of pure inner forms $\{G^t_0\}_{t\in I}$. As we will see, it is more natural to consider all the pure inner forms $\{G^t_0\}$ than to just consider $G_0$.

From the diagram above, we also see that the $F$-conjugacy class $<g>\in H^1(F,G)$ determines $G^g_0$ up to an isomorphism that is unique up to inner automorphisms defined over $\Fq$. In particular we see that the commutative $\Qlcl$-algebra $\Fun(G_0^g(\Fq))^{\Ggq}$ of class functions under convolution is canonically determined by $<g>\in H^1(F,G)$. Similarly the set $\Irrep(\Ggq)$ is also canonically determined by the $F$-conjugacy class of $g$.

Let us define the set 
\beq\label{defirrep}
\Irrep(G_0):=\coprod\limits_{<t>\in H^1(F,G)}{\Irrep(G^t_0(\Fq))}.
\eeq

%Let us define the category 
%\beq
%\Rep(G_0):=\bigoplus\limits_{<t>\in H^1(F,G)}\Rep(\Gtq)
%\eeq

\brk
The algebraic groups $G_0^{g}$ and $G_0^{h}$ defined over $\Fq$ may still be isomorphic even though $g,h\in G$ may lie in \emph{different} $F$-conjugacy classes. Even so, we regard $G^g_0$ and $G^h_0$ as \emph{distinct} pure inner forms in this case.
\erk

\brk
Note that if $G$ is connected, then there is only one $F$-conjugacy class and we  only need to consider the trivial pure inner form $G_0$.
\erk

\bprop\label{transport}(See \cite[Cor. 4.10]{Bo2}.)
For each pure inner form $G^g_0$ there is a natural equivalence $\D_{G_0}(G_0)\cong \D_{G^g_0}(G^g_0)$ of braided monoidal categories. If $C_0\in \D_{G_0}(G_0)$, we let $C_0^g$ denote the corresponding object in $\D_{G^g_0}(G^g_0)$ obtained using this transport functor.
\eprop

\subsubsection{$\D^F_G(G)$ as a triangulated category}\label{astc}
Let $\D_G^F(G)$ denote the category of $G$-equivariant $\Qlcl$-complexes on $G$ for the $F$-conjugation action of $G$ on itself.  By Lang's theorem each $F$-conjugacy class in $G$ is a union of certain connected components of $G$. Hence we have
\beq
\D_G^F(G)=\bigoplus\limits_{\O \in H^1(F,G)}\D_G^F(\O).
\eeq
For an $F$-conjugacy class $\O\subset G$, let us describe the category $\D_G^F(\O)$. The $F$-conjugation action on $\O$ is transitive, hence we have an equivalence $\D_G^F(\O)\cong\D_{\Stab^F_G(t)}(t)$ where $t$ is some point in $\O$ and $\Stab^F_G(t)$ is its stabilizer for the $F$-conjugation action. But $\Stab^F_G(t)=G^{ad(t)\circ F}=G_0^t(\Fq)$ is a finite group. Hence we have
\beq\label{Fconjorbit}
\D_G^F(\O)\cong D^b(\h{Rep}(G_0^t(\Fq))).
\eeq

Hence $\D_G^F(G)$ as a $\Qlcl$-linear triangulated category encodes the representation theory of the groups of $\Fq$-points of all the pure inner forms of $G_0$. The set $\Irrep(G_0)$ defined by (\ref{defirrep}) along with all their shifts give all the simple objects of the triangulated category $\D^F_G(G)$ (which also happens to be semisimple abelian). For $W\in \Irrep(\Gtq)\subset \Irrep(G_0)$, let $W_{loc}\in \D_G^F(G)$ denote the corresponding object. $W_{loc}$ is a local system supported on the $F$-conjugacy class $\O$ of $t$.

\subsubsection{$\D_G^F(G)$ as a $\D_G(G)$-module category}\label{modcat}
Let $C\in \D_G(G)$ and $M\in \D_G^F(G)$. Consider the action of $G$ on $G\times G$ given by conjugation of the first coordinate and $F$-conjugation of the second coordinate. Then $C\boxtimes M\in \D_G^{1,F}(G\times G)$, the equivariant derived category for this action of $G$ on $G\times G$. The multiplication map $\mu:G\times G\to G$ is $G$-equivariant where the action of $G$ on the latter $G$ is given by $F$-conjugation. Hence $C\ast M\in \D_G^F(G)$. Hence $\D_G^F(G)$ is a module category over the braided monoidal category $\D_G(G)$. Similarly we can prove that $M\ast C\in \D_G^F(G).$

Note that if $C\in \D_G(G)$, we have the braiding isomorphism $\b_{C,M}:C\ast M\to M\ast C$ for each $M\in \D(G).$ Moreover, if $M\in\D_G^F(G)$, we have the crossed braiding isomorphism $\b_{M,C}:M\ast C\to F^*(C)\ast M.$

\subsubsection{The semidirect product $G\rtimes\f{Z}$}\label{semidir}
Sometimes it is more convenient to think of the category $\D^F_G(G)$ in terms of the semidirect product $\t{G}:=G\rtimes \f{Z}$, where $1\in \f{Z}$ acts on $G$ by the Frobenius $F:G\to G$. We identify $\t{G}=\{gF^n|g\in G, n\in \f{Z}\}$ where $gF^n\cdot hF^m=gF^n(h)F^{n+m}$. Now $G$ acts on each coset $GF^n$ by conjugation.  We consider the category (see also \S\ref{brcrgrex})
$$\D_G(\t{G}):=\bigoplus\limits_{n\in \f{Z}}\D_G(GF^n).$$ It has the structure of a braided $\f{Z}$-crossed category. In particular we have a braided monoidal action of $\f{Z}$ on $\DG$ defined by $F^n(C):=(F^{-n})^*(C)$ for each $n\in \f{Z}$ and $C\in \DG$.

We have an identification 
\beq
\D^F_G(G)\cong\D_G(GF).
\eeq
Now $\D_G(GF)$ is also a $\DG$-bimodule category under convolution. For $M\in \D^F_G(G)$, we denote the corresponding object of $\D_G(GF)$ by $\t{M}$. Then for $C\in \DG$ and $M\in \D^F_G(G)$ we have natural isomorphisms
\beq\label{conv1}
\t{C\ast M}\cong C\ast \t{M},
\eeq
\beq\label{conv2}
\t{M\ast F(C)}\cong \t{M}\ast C \mbox{ or equivalently }\t{M\ast C}\cong \t{M}\ast F^*(C).
\eeq

With this identification, we will work interchangeably with the two equivalent categories $\D^F_G(G)$ and $\D_G(GF)$.

\subsubsection{A trace on $\D_G^F(G)$}\label{traceondgf}
For an object $M\in \D_G^F(G)$ and any endomorphism $f:M\to M$, we will define a trace $\h{tr}_F(f)\in \Qlcl.$ It suffices to consider $M\in \D_G^F(\O)\cong D^b(\h{Rep}(G_0^t(\Fq)))$ for some $\O\in H^1(F,G)$ and a point $t\in \O$.
Now the triangulated category $D^b(\h{Rep}(G_0^t(\Fq)))$ has the classical traces $tr(f)$ for endomorphisms $f:M\to M$. We rescale this classical trace and define 
\beq\label{deftr}
\tr_F(f)=\frac{tr(f)}{|G^t_0(\Fq)|}.
\eeq

We extend by additivity to define traces in the category $\D_G^F(G)$.

\subsubsection{The Frobenius algebra $\Fun([G_0](\Fq))$}\label{frobal}
Consider the subset $R\subset G\times G$ defined by 
\beq R=\{(x,g)|F(x)=g^{-1}xg\}=\{(x,g)|x\in G^g_0(\Fq)=\Stab_G^F(g)\}.\eeq
Clearly $(x,g)\in R$ if and only if $(x,xg)\in R$. 

If $(x,g)\in R$ and $h\in G$ then ${ }^h(x,g):=({ }^hx, hgF(h^{-1}))\in R$. Hence $R\subset G\times G$ is $G$-stable for the left action of $G$ on $G\times G$ described in \S\ref{modcat}. The number of $G$-orbits in $R$ is finite and the orbits correspond bijectively to the union of sets of conjugacy classes of the groups of $\Fq$-points of all inner forms of $G_0$ corresponding to $H^1(F,G)$. More precisely, each $G$ orbit in $R$ can be uniquely represented by a pair $(h,t)$ where $t\in G$ is a representative of an $F$-conjugacy class and $h\in \Gtq$ represents a conjugacy class in the inner form $\Gtq$ associated with $t\in G$.

Note that for $(x,g)\in R$
\beq\label{stab}
\Stab_G(x,g)=Z_{\Ggq}(x),
\eeq
the centralizer of $x$ in $\Ggq$.

\brk\label{Geq}
The map $(p_1,\mu):(x,g)\mapsto (x,xg)$ is $G$-equivariant. Hence for $(x,g)\in R$, we have
\beq\label{stabxg}
\Stab_G(x,g)=\Stab_G(x,xg).
\eeq
Also, a subset $\{(h,t)\}\subset R$ is a complete set of representatives of $G$-orbits in $R$ if and only if the subset $\{(h,h^{-1}t)\}\subset R$ is a complete set of representatives.
\erk

Consider the space of $G$-invariant $\Qlcl$-valued functions defined on the set $R$. We can identify this space with the space $\Fun([G_0](\Fq))$ defined in \cite[\S2.2]{Bo2}. In the terminology of {\it op. cit.}, $[G_0]$ denotes the quotient stack of $G_0$ by the conjugation action and the groupoid $[G_0](\Fq)$ can be identified with the disjoint union of the groupoids $[G^t_0(\Fq)]$ (see \cite[Example 2.8]{Bo2}) as $t$ ranges over a set of representatives of $F$-conjugacy classes in $G$.

The space $\Fun([G_0](\Fq))=\Fun_G(R)$ is an algebra under convolution.
 As an algebra it can be identified with the product of the convolution algebras of class functions on the pure inner forms $G^t_0(\Fq)$:
\beq
\Fun_G(R)=\FunG=\prod\limits_{<t>\in H^1(F,G)}{\Fun(\Gtq)^{\Gtq}}.
\eeq

 Moreover it is a Frobenius algebra with bilinear form defined by 
\beq
(f_1,f_2):=\sum\limits_{<(h,t)>\in G\backslash R}{\frac{f_1(h,t){f_2(h^{-1},t)}}{|\Stab_G(h,t)|}}.
\eeq
By Remark \ref{Geq} we see that
\beq
(f_1,f_2)=\sum\limits_{<(h,t)>\in G\backslash R}{\frac{f_1(h,h^{-1}t){f_2(h^{-1},h^{-1}t)}}{|\Stab_G(h,t)|}}.
\eeq
The corresponding linear functional $\lambda: \Fun([G_0](\Fq))\to \Qlcl$ is given by 
\beq\label{linfun}
\lambda(f)=\sum\limits_{<t>\in H^1(F,G)}\frac{f(1,t)}{|G^t_0(\Fq)|}.
\eeq

We also define a Hermitian inner form $<\cdot,\cdot>$ on the space $\Fun([G_0](\Fq))$. For this we fix an isomorphism $\Qlcl\to \f{C}$, whence we get a conjugation involution $\bar{(\cdot)}$ of the field $\Qlcl$ which maps roots of unity in $\Qlcl$ to their inverses. We define a Hermitian inner product on $\Fun([G_0](\Fq))$ by 
\beq
<f_1,f_2>:=\sum\limits_{<(h,t)>\in G\backslash R}{\frac{f_1(h,t)\bar{f_2(h,t)}}{|\Stab_G(h,t)|}}=\sum\limits_{<(h,t)>\in G\backslash R}{\frac{f_1(h,h^{-1}t)\bar{f_2(h,h^{-1}t)}}{|\Stab_G(h,t)|}}.
\eeq

By taking characters of irreducible representations, we consider the set $\Irrep(G_0)$ defined by (\ref{defirrep}) as a subset of $\FunG$. Using the orthogonality relations for irreducible characters we get:
\bprop\label{charofirrep}
The set $\Irrep(G_0)\subset \FunG$ is an orthonormal basis of $\FunG$ with respect to  both the Hermitian inner product $<\cdot,\cdot>$ as well as the bilinear inner product $(\cdot,\cdot)$.
\eprop

\subsubsection{The character of an object of $\D_G^F(G)$}\label{cofdfg}
If $M\in \D_G^F(G)$, we will define its character $\chi_M$ as a $G$-invariant function from the set $R$ to $\Qlcl$. If $M\in \D_G^F(G)$ and $g\in G$, the stalk $M_g$ can be considered as an object of $D^b(\h{Rep}(G^g_0(\Fq)))$ and we have its character $\chi_{M_g}:G^g_0(\Fq)\to \Qlcl.$ If $(x,g)\in R$, we define $\chi_M(x,g):=\chi_{M_g}(x).$ 

Equivalently, for $M\in \D_G^F(G)$ and $(x,g)\in G\times G$, we have an isomorphism $\phi_M(x,g): M_g\to M_{xgF(x)^{-1}}$, where $\phi_M$ is the equivariant structure associated to $M\in \D_G^F(G)$. For $(x,g)\in R$, $\phi_M(x,g)$ is an automorphism of $M_g$ and we have 
\beq\label{chim}\chi_M(x,g):=tr(\phi_M(x,g)).\eeq 
It is easy to check that $\chi_M:R\to \Qlcl$ is $G$-invariant.

Hence given any object $M\in \D^F_G(G)$, we have its character $\chi_M\in \Fun([G_0](\Fq))$.

Now if $W\in \Irrep(G_0)$, we can take its character $\chi_W\in \FunG$. It is clear from the definition that $\chi_W$ is the same as $\chi_{W_{loc}}$ as defined in this subsection.

%\subsubsection{The characters as  ``trace of Frobenius'' functions}
%In this section we construct a functor 
%\beq\D_G^F(G)\to \bigoplus\limits_{<t>\in H^1(F,G)}\D^{Weil}(G^t_0)\eeq
%and interpret the characters defined in \S\ref{cofdfg} using this functor.  

\subsubsection{The ``trace of Frobenius'' function of an equivariant Weil sheaf}\label{sfcorr}
Consider the category $\D_{G_0}^{Weil}(G_0)$ whose objects are pairs $(C,\psi)$ consisting of an object $C\in \D_G(G)$ and an isomorphism $\psi:F^*(C)\to C$  in $\D_G(G)$. Let $(C,\psi)\in \D^{Weil}_{G_0}(G_0)$. Let $g\in G$ and consider the inner form $G^g_0$ defined over $\Fq$. Using the sheaf-function correspondence we define the associated trace function 
$$t^g_{C,\psi}:G^g_0(\F_q)\to \Qlcl.$$ 
We define the function $T_{C,\psi}:R\to \Qlcl$ by $T_{C,\psi}(x,g):=t^g_{C,\psi}(x).$

More precisely, for $C\in \D_G(G)$, $x,g\in G$, we have isomorphisms $$\phi_C(g^{-1}, gF(x)g^{-1}):C_{gF(x)g^{-1}}\to C_{F(x)}.$$ If $\psi:F^*M\xto{\cong} M$ is a Weil structure, we have isomorphisms $\psi(x):C_{F(x)}\to C_x$. If $(x,g)\in R$, then $\psi(x)\circ\phi_C(g^{-1}, gF(x)g^{-1})$ is an automorphism of $C_x$ and we define
\beq
T_{C,\psi}(x,g):=tr(\psi(x)\circ\phi_C(g^{-1}, gF(x)g^{-1})).
\eeq
The function $T_{C,\psi}:R\to \Qlcl$ is also $G$-invariant.

Hence given an object $(C,\psi)\in \D_{G_0}^{Weil}(G_0)$, we have its trace function $T_{C,\psi}\in \FunG$.

For each $g\in G$, we have natural functors $\D_{G^g_0}(G^g_0)\to \D^{Weil}_{G_0}(G_0)$ that are compatible with the identification in Proposition \ref{transport}. If $C_0\in \D_{G_0}(G_0)$, we let $T_{C_0}\in \FunG$ denote the corresponding function.

Using Lemma \ref{compwithconv} we obtain:
\blem\label{eqcompwithconv}
Let $C,D\in \D^{Weil}_{G_0}(G_0)$. Then $T_C\ast T_D=T_{C\ast D}$.
\elem

\subsubsection{The inner product formula}\label{secipf}

Suppose that $M\in \D_G^F(G)$ and $(C,\psi)\in\D^{Weil}_{G_0}(G_0)$. Then we have the automorphism
\beq
\g_{C,\psi,M}:C\ast M \xto{\beta_{C,M}} M\ast C\xto{\beta_{M,C}} F^*(C)\ast M\xto{\psi\ast \id_M} C\ast M
\eeq
in $\D_G^F(G)$.

\brk
If we use the identification $\D^F_G(G)\cong \D_G(GF)$ (see \S\ref{semidir}), then using (\ref{conv1}),(\ref{conv2}) we see that $\g_{C,\psi,M}$ as defined above gets identified with $\g_{C,\psi,\t{M}}$ as defined in \S\ref{bcgandcsm}.
\erk

In our first main result, we describe the inner product of the functions associated with $(C,\psi)$ and $M$.
\bthm\label{ipf}
For $(C,\psi)\in \D_{G_0}^{Weil}(G_0)$ and $M\in \D_G^F(G)$ the inner product of the functions $T_{C,\psi}, \chi_M\in \FunG$ is given by
\beq
(T_{C,\psi},\chi_M)=<T_{C,\psi},\chi_M>= \tr_F(\g_{C,\psi,M}).
\eeq
\ethm
We will prove this result in \S\ref{pfofipf} using the Grothendieck-Lefschetz trace formula.

\subsection{Neutrally unipotent groups}
We now restrict our attention to neutrally unipotent groups $G$, namely groups whose neutral connected component $G^\circ$ is unipotent. The theory of character sheaves for such groups was developed by Boyarchenko and Drinfeld in \cite{BD1}, \cite{BD2}, \cite{Bo1}, \cite{Bo2}. We briefly recall some aspects of this theory in this subsection and state our main results and conjectures.

\subsubsection{Character sheaves on neutrally unipotent groups}\label{csnug}
Let $G$ be a neutrally unipotent group defined over $\k$. Character sheaves on $G$ are defined in terms of minimal idempotents in the braided monoidal category $\DG$. To define character sheaves we do not need an $\Fq$-structure on $G$. We recall the definition of character sheaves in Definition \ref{CSdef} below. 

An idempotent in $\DG$ is an object $e\in \DG$ such that $e\ast e\cong e.$ An idempotent $e$ is said to be minimal if $e\neq 0$ and for any idempotent $e'$ in $\DG$, either $e\ast e'=0$ or $e\ast e'\cong e.$ An idempotent $e$ is said to be closed if there exists an arrow $\delta_1\to e$ in $\DG$ which becomes an isomorphism after convolution with $e$.

The following statements are proved in \cite{BD2}, \cite{De1} and \cite{De2}:
\bthm\label{BDmain} Let $G$ be a neutrally unipotent group and let $e$ be a minimal idempotent in $\DG$. Then we have the following:\\
(i) The idempotent $e$ is in fact a closed idempotent and hence the Hecke subcategory $e\DG$ is a monoidal category with unit object $e$.\\
(ii) Let $\M_{G,e}^{perv}\subset e\DG$ be the full subcategory of those objects whose underlying $\Qlcl$-complex is a perverse sheaf on $G$. Then $\M_{G,e}^{perv}$ is a semisimple abelian category with finitely many simple objects and $e\DG\cong D^b(\M_{G,e}^{perv})$.\\
(iii) There exists an integer $n_e\in \{0,1,\cdots,\dim G\}$ such that $e[-n_e]\in \M_{G,e}^{perv}$. This means that every minimal idempotent $e\in \DG$ is a perverse sheaf up to a shift. Moreover, we have $$\f{D}^-e\cong e[-2n_e](-n_e)$$ in $\DG$, where $\f{D}^-$ is the dualization functor introduced in \S\ref{nac}.\\
(iv) The category $\Meg:=\M_{G,e}^{perv}[n_e]$ is closed under convolution. In fact $\Meg$ is a modular category with unit object $e$. The ribbon structure on $\Meg$ comes from the natural ribbon $\mathfrak{r}$-category structure on $\DG$.\\
(v) We have $\dim \Meg=\FPdim \Meg=m^2$ for some integer $m\in \f{Z}$, where $\dim$ and $\FPdim$ denote the categorical dimension and the Frobenius Perron dimension respectively. The modular category $\Meg$ is integral i.e. the categorical dimensions of all objects are integers. 
\ethm

\bdefn\label{CSdef}
(i) For a minimal idempotent $e\in \DG$, the number $d_e:=\frac{\dim G - n_e}{2}$ is called the functional dimension of $e$.\\
(ii) Let $e$ be a minimal idempotent in $\DG$. We define the $\f{L}$-packet of character sheaves on $G$ associated with the minimal idempotent $e$ to be the following (finite) set 
\beq
CS_e(G):=\{\mbox{isom. classes of simple objects of $\Meg\subset \DG$}\}.
\eeq
(iii) If $e'\in \DG$ is a minimal idempotent which is not isomorphic to $e$, then the sets $CS_e(G)$ and $CS_{e'}(G)$ are disjoint. Let $\widehat{G}$ denote the set of isomorphism classes of minimal idempotents in $\DG$. The set of character sheaves on $G$ is defined to be the disjoint union of all $\f{L}$-packets of character sheaves:
\beq
CS(G):=\coprod\limits_{e\in \widehat{G}} CS_e(G).
\eeq
\edefn

So far we have defined character sheaves using the rather abstract notion of minimal idempotents in $\DG$. We now give a more explicit description of minimal idempotents in $\DG$ using the notion of admissible pairs for $G$. 

We first briefly recall the notion of Serre duality of connected unipotent groups. Our assumption of perfectness is essential in the theory of Serre duality. We refer to \cite[\S A]{Bo1} for a detailed exposition. Let $U$ be a connected unipotent group. Then we have its Serre dual $U^*$ which is a (perfect) commutative (possibly disconnected) unipotent group. Roughly speaking, $U^*$ is the moduli space of central extensions of $U$ by the discrete group scheme $\Q_p/\f{Z}_p$. Let us fix an identification of $\Q_p/\f{Z}_p$ with the group $\mu_{p^\infty}(\Qlcl^\times)$ of the $p$-power-th roots of unity in $\Qlcl^\times$. Then by \cite[Lemma 7.3]{Bo1}, we may think of $U^*$ as the moduli space of multiplicative local systems on $U$.

\brk\label{compidentification}
 We have a natural identification $\Q_p/\f{Z}_p\cong\mu_{p^\infty}(\f{C})$ where we identify $\frac{1}{p^n}\in \Q_p/\f{Z}_p$ with $e^{\frac{2\pi i}{p^n}}$. Since we have already chosen an identification $\Qlcl\cong \f{C}$ in \S\ref{frobal}, this also determines the identification  $\Q_p/\f{Z}_p\cong\mu_{p^\infty}(\Qlcl)$. Henceforth we assume that these two identifications are compatible in this sense.
\erk

\bdefn\label{defadpair} (\cite[Def. 7.4]{Bo1})
Let $G$ be a neutrally unipotent group. Let $(H,\L)$ be a pair consisting of a connected subgroup $H\subset G$ and a multiplicative local system $\L$ on $H$. We say that the pair $(H,\L)$ is admissible for $G$ if the following conditions hold:
\bit
\item[(i)] Let $G'\subset G$ denote the normalizer of the pair $(H,\L)$ (see \cite[\S7.3]{Bo1}) and let $G'^\circ$ denote its neutral connected component. Then $G'^{\circ}/H$ should be commutative.
\item[(ii)] The group morphism $\phi_\L:G'^{\circ}/H \to (G'^{\circ}/H)^*$ that is defined in this situation (see \cite[\S A.13]{Bo1}) should be an isogeny.
\item[(iii)] {\it(Geometric Mackey condition)} For every $g\in G-G'$, we should have $$\L|_{(H\cap ^gH)^\circ}\ncong {  }^g\L|_{(H\cap ^gH)^\circ},$$ where $^gH=gHg^{-1}$ and $^g\L$ is the multiplicative local system on ${}^gH$ obtained from $\L$ by transport of structure. 
\eit
\edefn

\brk
In the situation above, let $e_{H,\L}:=\L\otimes \can_H\in \D_{G'}(G')$. Then $e_{H,\L}\in \D_{G'}(G')$ is in fact a minimal idempotent (see \cite{BD2}). 
\erk

\bdefn
If $(H,\L)$ is an admissible pair for $G$ such that $G'=G$, we say that $(H,\L)$ is a Heisenberg admissible pair  for $G$. Note that in this situation, condition (iii) of the definition is vacuous. In this case the minimal idempotent $e_{H,\L}\in \D_G(G)$ is said to be a Heisenberg idempotent. 
\edefn

\noin In \cite{BD2}, the induction (with compact support) functor $\h{ind}_{G'}^G:\D_{G'}(G')\to \D_G(G)$ is defined for closed subgroups $G'\subset G$. In the following theorem, Boyarchenko and Drinfeld prove that every minimal idempotent in $\D_G(G)$ comes from an admissible pair and in particular from a Heisenberg idempotent on a subgroup by induction.

\bthm\label{BDmain2}
(\cite{BD2})
(i) Let $(H,\L)$ be an admissible pair for a neutrally unipotent group $G$ and let $e_{H,\L}\in \D_{G'}(G')$ be the corresponding Heisenberg idempotent on $G'$ as defined above. Then $f_{H,\L}:=\ind_{G'}^Ge_{H,\L}\in \D_G(G)$ is a minimal idempotent. The functional dimensions of $e_{H,\L}\in \DGp$ and $f_{H,\L}\in \DG$ are related as follows:$$d_{G,f_{H,\L}}=d_{G',e_{H,\L}}+\dim(G/G').$$\\
(ii) In the situation of (i), the induction functor induces an equivalence of modular categories 
$$\indg:\M_{G',e_{H,\L}}\xto{\cong} \M_{G,f_{H,\L}}.$$\\
(iii) Every minimal idempotent $f\in \D_G(G)$ comes from an admissible pair by the procedure described in (i). Hence every minimal idempotent $f\in \D_G(G)$ comes from induction from a Heisenberg idempotent $e$ on some subgroup $G'$.
\ethm

This result reduces many problems about general minimal idempotents to the case of Heisenberg idempotents.

\subsubsection{$\f{L}$-packets of characters of neutrally unipotent groups}
From now on we will need the $\Fq$-structure on $G$ provided by a Frobenius map $F:G\to G$. The functor $F:={F^{-1}}^*:\DG\to \DG$ is an autoequivalence of ribbon $\mathfrak{r}$-categories as we have already mentioned in \S\ref{semidir}. This functor gives an action of $F$ on the set $\widehat{G}$ of minimal idempotents as well as on the set $CS(G)$ of character sheaves. We see that $F$ maps the set $CS_e(G)$ to the set  $CS_{F(e)}(G)$. We are interested in the set $CS(G)^F$ of $F$-stable character sheaves. Clearly we have 
\beq
CS(G)^F=\coprod\limits_{e\in \widehat{G}^F}CS_e(G)^F.
\eeq

The following result from \cite{Bo2} describes the set $\widehat{G}^F$.
\bthm\label{Fstabe} (See \cite[Thm. 2.17]{Bo2}.) Let $e\in \widehat{G}^F$, i.e. $e$ a minimal idempotent in $\DG$ such that $F^*e\cong e$. Then:\\
(i) Consider a pure inner form $G_0^g$ of $G_0$. There exists a unique (up to isomorphism) weak idempotent $e^g_0\in \D_{G^g_0}(G^g_0)$ such that $e$ is obtained from $e^g_0$ by base change. This idempotent $e^g_0$ is in fact a closed idempotent in $\D_{G^g_0}(G^g_0)$.\\
(ii) There exists an inner form $G^g_0$ such that the closed idempotent $e^g_0\in \D_{G^g_0}(G^g_0)$ is obtained from an admissible pair defined over $\Fq$ for the inner form $G^g_0$.
\ethm

\brk
In general in the situation of the theorem, $e_0\in \D_{G_0}(G_0)$ may not be defined by an admissible pair for $G_0$ and passing to an inner form is necessary.
\erk

\bdefn
An idempotent $e_0\in \D_{G_0}(G_0)$ is said to be a geometrically minimal idempotent if the idempotent $e\in \DG$ obtained by base change is a minimal idempotent. It is clear that if $e_0\in \D_{G_0}(G_0)$ is a geometrically minimal idempotent then for each $g\in G$, $e_0^g\in \D_{G^g_0}(G^g_0)$ is also a geometrically minimal idempotent. By Theorem \ref{Fstabe}, there is a natural bijection between the set $\widehat{G}^F$ and the set of isomorphism classes of geometrically minimal idempotents in $\D_{G_0}(G_0)$.
\edefn

\bdefn\label{lpacketchar} (See \cite[Defn. 2.10]{Bo2}.)
Let $e_0\in \D_{G_0}(G_0)$ be a geometrically minimal idempotent. For each $g\in G$, let $\Irrep_{e^g_0}(G_0^g(\Fq))$ be the set of isomorphism classes of irreducible representations of $G_0^g(\Fq)$ in which the idempotent $t_{e_0^g}\in \Fun(G^g_0(\Fq))^{G_0^g(\Fq)}$ acts as the identity. The $\f{L}$-packet of irreducible representations of $G_0$ defined by $e_0$ is defined to be the set
\beq
\Irrep_{e_0}(G_0):=\coprod\limits_{<t>\in H^1(F,G)}{\Irrep_{e^t_0}(G^t_0(\Fq))}\subset \Irrep(G_0).
\eeq
By taking the characters of representations, we consider the set $\Irrep_{e_0}(G_0)$ as a subset of the Frobenius algebra $\FunG$. Then the set $\Irrep_{e_0}(G_0)\subset \FunG$ forms an orthonormal basis for the subspace $T_{e_0}\FunG\subset \FunG$  corresponding to the idempotent $T_{e_0}\in \FunG$.
\edefn

\subsubsection{$F$-stable character sheaves on neutrally unipotent groups}
Let $e\in \widehat{G}^F$. Then the equivalence $F:\DG\to\DG$ induces an autoequivalence of the modular category $F:\M_{G,e}\to \Meg$ and a permutation of the $\f{L}$-packet $CS_e(G)$ of character sheaves associated with $e$. The following is proved in \cite{Bo2}.
\bthm\label{Bo2main}(See \cite[Thm. 2.17]{Bo2}.)
(i) Let $C\in CS_e(G)^F$ and let $\psi:F^*C\to C$ be any Weil structure. Then $T_{C,\psi}\in T_{e_0}\FunG$.\\
(ii)  For each $C\in CS_e(G)^F$ choose a Weil structure $\psi_C$ so that $<T_{C,\psi_C},T_{C,\psi_C}>=1.$ Then the finite set $\{T_{C,\psi_C}\}_{C\in CS_e(G)^F}$ is an orthonormal basis of $T_{e_0}\FunG$ (with respect to the Hermitian inner product $<\cdot,\cdot>$).
\ethm

This theorem proves that the trace functions $\{T_{C,\psi}\}_{C\in CS(G)^F}$ (suitably normalized) also form an orthonormal basis for $\FunG$, that the matrix relating this basis to the basis (formed by characters of) $\Irrep(G_0)$ is block diagonal (see also Definition \ref{lpacketchar}) and that these blocks correspond to $F$-stable minimal idempotents $e\in \DG$ or equivalently $F$-stable $\f{L}$-packets of $G$.

\subsubsection{Main results and conjectures}\label{mrc}
As before, let $G$ be a neutrally unipotent group equipped with an $\Fq$-structure. Let $e\in\widehat{G}^F$ and let $e_0\in \DGn$ be the corresponding geometrically minimal idempotent. Note that the subspace $T_{e_0}\FunG\subset \FunG$ is an algebra with unit $T_{e_0}$. As seen before, the set $\Irrep_{e_0}(G_0)$ is an orthonormal basis of $T_{e_0}\FunG$ with respect to the bilinear form $(\cdot,\cdot)$ (as well as the Hermitian form $<\cdot,\cdot>$). Hence the restriction of the bilinear form to $T_{e_0}\FunG$ is non-degenerate and hence $T_{e_0}\FunG$ is a Frobenius algebra. (More generally, it is easy to see that if $A$ is a Frobenius algebra and $e\in A$ is a central idempotent, then $eA$ is again a Frobenius algebra.) Now, the suitably normalized trace of Frobenius functions $\{T_{C,\psi_C}\}_{C\in CS_e(G)^F}$ form another orthonormal basis of $T_{e_0}\FunG$ with respect to the Hermitian inner product $<\cdot,\cdot>$. Our goal is to describe the unitary matrix $\t{S}^{G_0,e_0}$ which relates these two orthonormal bases. Since the basis $\Irrep_{e_0}(G_0)$ is orthonormal, this matrix is obtained by taking inner products between the bases: 
\beq\label{tsmatrix}
(\t{S}^{G_0,e_0}_{C,W})_{C\in CS_e(G)^F, W\in \Irrep_{e_0}(G_0)}=(<T_{C,\psi_C},\chi_W>)_{C\in CS_e(G)^F, W\in \Irrep_{e_0}(G_0)}.
\eeq
First we give another characterization of the set $\Irrep_{e_0}(G_0)$ of irreducible representations of the various inner forms lying in the $\f{L}$-packet defined by $e_0$.
\bthm\label{main2}
Let $\M_{GF,e}\subset e\D_G(GF)$ be the full subcategory formed by objects whose underlying $\Qlcl$-complex is a perverse sheaf shifted by $n_e$. (Recall the definition of $n_e, d_e$ from \S\ref{csnug}.)\\ 
(i) Then $\M_{GF,e}$ has a natural structure of an invertible $\M_{G,e}$-module category equipped with a natural normalized (see (\ref{normtr})) $\Meg$-module trace  $\tr_{F,e}$ and we have $D^b(\M_{GF,e})\cong e\D_G(GF)$. \\
(ii)  Let $W\in \Irrep_{e_0}(G_0)$ i.e. $W\in \Irrep_{e_0^t}(\Gtq)$ for some $t\in G$. Then by (\ref{Fconjorbit}), $W$ corresponds to a $G$-equivariant local system $W_{loc}\in \D_G(GF)$ supported on the $G$-conjugacy class of $tF\subset GF$. Then in fact $W_{loc}\in e\D_G(GF)$ and $M_W:=W_{loc}[\dim G+n_e]\in \M_{GF,e}$ is a simple object.\\
(iii) There is a natural bijection between the set $\Irrep_{e_0}(G_0)$ and the set $\O_{\M_{GF,e}}$ of (isomorphism classes of) simple objects of $\M_{GF,e}$. This bijection is induced by the map 
$$\Irrep_{e_0}(G_0)\ni W\mapsto M_W=W_{loc}[\dim G+n_e]\in \O_{\M_{GF,e}}$$ 
from (ii). 
\ethm

\brk\label{simplebasis}
By taking characters ($\Irrep(G_0)\ni W\mapsto \chi_W\in \FunG$), we have identified the set $\Irrep(G_0)$ as a subset of $\FunG$. We see that $\chi_W=(-1)^{\dim G+n_e}\cdot\chi_{M_W}=(-1)^{2d_e}\cdot\chi_{M_W}$.
\erk

Using Theorem \ref{Fstabe}(ii) we will first reduce Theorem \ref{main2} to the case of Heisenberg idempotents in \S\ref{redofmain2}. We then prove the theorem in the Heisenberg case in \S\ref{pfofmain2}.

Combining the Remark \ref{simplebasis} with (\ref{tsmatrix}) and Theorem \ref{ipf} we obtain:
\bcor\label{cor1}
The entries of the matrix $\t{S}^{G_0,e_0}$ are 
$$\t{S}^{G_0,e_0}_{C,W}=(-1)^{2d_e}\cdot\tr_F(\g_{C,\psi_C,M_W})$$
where $\g_{C,\psi_C,M_W}$ is the automorphism of $C\ast M_W\in \D_G(GF)$ defined in \S\ref{secipf} and $\tr_F$ is the  trace on $\D_G(GF)$ defined in \S\ref{traceondgf}.
\ecor

Now the full subcategory $\M_{GF,e}\subset \D_G(GF)$ has the natural module trace $\tr_{F,e}$ from Theorem \ref{main2}(i). We would like to compare this module trace with the trace $\tr_F$ on $\D_G(GF)$ restricted to $\M_{GF,e}$.

If $G$ is a neutrally unipotent group, it is known that the modular category $\M_{G,e}$ is integral. In fact, in a subsequent work (\cite{De4}), we prove that $\Meg$ is in fact positive integral, at least for $G$ defined over $\Fqcl$ as in this paper. Hence for any  $C\in \M_{G,e}$ we have $\dim(C)=\FPdim(C)\in \f{Z}_{>0}$ and $\dim(\Meg)=\FPdim(\Meg)\in \Z_{>0}$. Now let $\M_{GF,e}^+$ denote the category which is same as $\M_{GF,e}$ as a module category but which is equipped with the positive module trace (normalized according to (\ref{normtr})) $\tr^+_{F,e}$, i.e. any object in $\M_{GF,e}^+$ has positive (with respect to our chosen identification $\Qlcl\cong \f{C}$)  categorical dimension (which may be non-integral). Recall from \cite[\S2.5]{ENO} the notion of Frobenius-Perron dimension in module categories. Then for any $M\in \M_{GF,e}^+$ we have (cf. \cite[Prop. 2.2]{ENO} and (\ref{normtr})) $\dim(M)=\FPdim(M)$. Note that we have $\tr^+_{F,e}=\pm\tr_{F,e}$

The next result compares the $\Meg$-module trace $\tr_{F,e}^+$ on $\M_{GF,e}^+$ with the trace $\tr_F$ restricted to $\M_{GF,e}\subset \D_G(GF)$ and gives an interpretation of the matrix $\t{S}^{G_0,e_0}$ defined in (\ref{tsmatrix}) as a crossed $S$-matrix up to scaling.
\bthm\label{main3}
(i) On the subcategory $\M_{GF,e}\subset \D_G(GF)$ we have the following equality:
\beq
\tr_{F,e}^+=(-1)^{2d_e}\cdot\frac{q^{\dim G}\cdot\sqrt{\dim \Meg}}{q^{d_e}}\cdot \tr_F.
\eeq
(ii) The matrix $\t{S}^{G_0,e_0}$ relating the two orthonormal bases $\{T_{C,\psi_C}\}_{C\in CS_e(G)^F}$ and $\Irrep_{e_0}(G_0)$ of $T_{e_0}\FunG$  is equal to the crossed $S$-matrix corresponding to the modular category $\Meg$ and the $\Meg$-module category $\M^+_{GF,e}$  up to a scaling:
\beq\label{smat+}
\t{S}^{G_0,e_0}=\frac{q^{d_e}}{q^{\dim G}\cdot\sqrt{\dim \M_{G,e}}}\cdot S^{\Meg, \M_{GF,e}^+}=\pm\frac{q^{d_e}}{q^{\dim G}\cdot\sqrt{\dim \M_{G,e}}}\cdot S^{\Meg, \M_{GF,e}}.
\eeq
\ethm

\brk
By Theorem \ref{BDmain}(v), the categorical dimension $\dim \Meg$ is the square of an integer and we extract the positive integral square root in the theorem above. Also $d_e\in \frac{1}{2}\f{Z}$ and we again extract the positive square root in the formulas above if $d_e\in \frac{1}{2}+\f{Z}$ according to our chosen identification $\Qlcl\cong \f{C}$.
\erk

Once again we will first reduce Theorem \ref{main3} to the case of Heisenberg idempotents in \S\ref{redofmain3} and prove it in the Heisenberg case in \S\ref{pfofmain3}.

As a corollary of the above results, we obtain the following result about dimensions of irreducible representations (see also Corollary \ref{c:sqdim} for an alternative approach):
\bthm\label{t:corthm}
Let $G$ be a neutrally unipotent group and suppose that $e\in \hat{G}^F$. Suppose that $W\in \Irrep_{e_0}(G_0)$ is an irreducible representation of the pure inner form $\Gtq$. Let $M_W\in \O_{\M_{GF,e}}$ be the corresponding simple object as defined by Theorem \ref{main2}(iii). Then we have 
\beq\label{e:dimw}
\dim(W)=\frac{\dim_{\M^+_{GF,e}}(M_W)\cdot q^{d_e}\cdot |\Pi_0(G_0^t)(\Fq)|}{\sqrt{\dim \Meg}}=\frac{\dim_{\M^+_{GF,e}}(M_W)\cdot q^{d_e}\cdot |G_0^t(\Fq)|}{q^{\dim G}\cdot\sqrt{\dim \Meg}},
\eeq
where $\dim_{\M^+_{GF,e}}(M_W)$ denotes the categorical dimension of $M_W$ in $\M^+_{GF,e}$. Furthermore, we have
\beq\label{e:sumsq}
\sum\limits_{W\in \Irrep_{e_0}(G_0)}\frac{\dim(W)^2}{|\Gtq|^2}=\frac{1}{q^{\dim G+n_e}}=\frac{q^{2d_e}}{q^{2\dim G}}.
\eeq
In particular if $G$ is connected, then
\beq\label{e:sumsqconn}
\sum\limits_{W\in \Irrep_{e_0}(G_0(\Fq))}{\dim(W)^2}={q^{2d_e}}.
\eeq
For any neutrally unipotent group $G$ equipped with an $\Fq$-Frobenius $F:G\rar{}G$, we have
\beq\label{e:fundimsq}
\sum\limits_{e\in \hat{G}^F}q^{2d_e}=q^{\dim G}.
\eeq
\ethm
\bpf
The statement (\ref{e:dimw}) follows directly from Theorem \ref{main3}(i) and the definition of the simple object $M_W=W_{loc}[\dim G+n_e]$. The statement (\ref{e:sumsq}) follows from (\ref{e:dimw}) using the fact that the trace $\tr^+_{F,e}$ was normalized according to (\ref{normtr}). If $G$ is connected, then there is only one pure inner form $G_0$ and $G_0(\Fq)=q^{\dim G}$, which proves (\ref{e:sumsqconn}).

Finally by using (\ref{e:sumsq}) for each $e\in \hat{G}^F$ and then taking the sum, we obtain
\beq
\sum\limits_{e\in \hat{G}^F}q^{2d_e}=q^{2\dim G}\cdot\sum\limits_{W\in \Irrep(G_0)}\frac{\dim(W)^2}{|\Gtq|^2}.
\eeq
\beq
=q^{2\dim G}\cdot\sum\limits_{\substack{<t>\\\in H^1(F,G)}}\frac{1}{|\Gtq|^2}\cdot\sum\limits_{W\in \Irrep(G^t_0(\Fq))}{\dim(W)^2}.
\eeq
\beq
=q^{2\dim G}\cdot\sum\limits_{<t>\in H^1(F,G)}\frac{1}{|\Gtq|}
\eeq
\beq
=q^{\dim G}\cdot\sum\limits_{<t>\in H^1(F,\Pi_0(G))}\frac{1}{|\Pi_0(G_0^t)(\Fq)|}
\eeq
\beq
=q^{\dim G}.
\eeq
\epf

Finally we state the following conjecture.
\bconj\label{conjmain}
(i) If $G$ is a neutrally unipotent group over any algebraically closed field $\k$ of characteristic $p$ and $e\in \DG$ is any minimal idempotent, then the modular category $\Meg$ is \emph{positive} integral. (This is proved for $\k=\Fqcl$ in \cite{De4}.)\\ 
(ii) Now assume as before that $\k=\Fqcl$ we have an $\Fq$-structure on $G$ and that $e\in \DG$ is an $F$-stable minimal idempotent. Then we can choose an identification $\Qlcl\cong \f{C}$ (compatible with any chosen identification $\Q_p/\f{Z}_p\cong\mu_{p^\infty}(\Qlcl)$ in the sense of Remark \ref{compidentification}) in such a way that the $\Meg$-module trace $\tr_{F,e}$ on $\M_{GF,e}$ is positive, i.e. we can choose an identification $\Qlcl\cong \f{C}$ so that $\M_{GF,e}^+=\M_{GF,e}$.
\econj

\subsection{Examples}\label{s:examples}
We will study two examples which illustrate the relationship between character sheaves and characters as described in this paper and see some interesting examples of crossed $S$-matrices and nontrivial $\f{L}$-packets.

\subsubsection{Finite groups}
We first consider the simple yet very instructive examples of zero dimensional algebraic groups. In other words, let $G$ be any finite group considered as a zero dimensional algebraic group over $\k$. An $\Fq$-structure on $G$ is determined by any automorphism $F:G\rar{\cong}G$ of the group. The category $\Sh_G(G)$ of $G$-equivariant $\Qlcl$-sheaves is precisely the modular category which is known as the Drinfeld double of the finite group $G$. In this case $\DG=D^b(\Sh_G(G))$ has only one minimal idempotent namely the unit object and hence we have a unique $\f{L}$-packet. In this case the finite groups that we are interested in are the groups $G^{\h{ad}(t)\circ F}$ parametrized by $<t>\in H^1(F,G)$. In this case 
\beq
CS(G)=\{(x,\rho)|x\in G, \rho\in\Irrep(C_G(x))\}/G\mbox{-conj}. 
\eeq
Given an isomorphism $\psi$ between the simple objects corresponding to $F^*(x,\rho)=(F^{-1}(x), \rho\circ F)$ and $(x,\rho)$ we obtain the trace of Frobenius function $T_{x,\rho,\psi}\in \FunG=\Fun_G(R)$.
On the other hand, we have the set of irreducible representations of all pure inner forms
\beq
\Irrep(G,F)=\{(tF, \xi)|tF\in GF, \xi\in \Irrep(G^{\h{ad}(t)\circ F})\}/G\mbox{-conj}.
\eeq
Given $(tF,\xi)\in \Irrep(G,F)$ we have its character $\chi_{tF,\chi}\in \FunG=\Fun_G(R)$ supported on the pure inner form $G^{\h{ad}(t)\circ F}$.
In this case the character sheaves and irreducible characters are related by the $CS(G)^F\times \Irrep(G,F)$ crossed $S$-matrix $S(G,F)$ whose entries are given by
\beq
S(G,F)_{x,\rho,\psi,tF,\xi}=\frac{1}{|C_G(x)|\cdot|G^{\h{ad}(t)\circ F}|}\sum\limits_{h:hxh^{-1}\in G^{\h{ad}(t)\circ F}}T_{x,\rho,\psi}({hxh^{-1}},t)\cdot\chi_{tF,\xi}(hx^{-1}h^{-1},t)
\eeq
\beq
=\frac{1}{|C_G(x)|\cdot|G^{\h{ad}(t)\circ F}|}\sum\limits_{h:hxh^{-1}\in G^{\h{ad}(t)\circ F}}T_{x,\rho,\psi}({hxh^{-1}},t)\cdot\overline{\chi_{tF,\xi}(hxh^{-1},t)}.
\eeq
Such crossed $S$-matrices associated with any automorphism of a finite group were also studied by Lusztig in his work on character sheaves on reductive groups under the name ``nonabelian Fourier transform''.

\subsubsection{The fake Heisenberg group}
Let $U$ be the fake Heisenberg group, namely $U=\G_a\times \G_a$ as a scheme with multiplication defined by $(x,a)\cdot (y,b)=(x+y,a+b+xy^p)$. Equivalently $U=\left\{\left(\begin{array}{ccc} 1 & x & a\\ & 1 & x^p\\ & & 1\end{array}\right)\right\}$. Consider the subgroup $\G_a\cong H=\{(0,a)\}\subset U$. We have the exact sequence
\beq
0\to H\to U\rar{\pi} \G_a\rar{} 0.
\eeq

We fix a nontrivial character $\chi:\Fp\to \Qlcl^\times$. This allows us to identify $\G_a$ with its Serre dual $\G_a^*$. We can then label the multiplicative local systems on $\G_a$ as $\L_s$ where $s\in \G_a$. Then we have the following admissible pairs and minimal idempotents for $U$:
\bit
\item Pairs of the form $(U,\pi^*\L_s)$ for any $s\in \G_a$. The corresponding minimal idempotents are $e_{U,\pi^*\L_s}$. 
\item Pairs of the form $(H,\L_s)$ for any nonzero $s\in \G_a$. The corresponding minimal idempotents are $e_{H,\L_s}$.
\eit

Now let $q=p^m$ and let $F:U\to U$ denote the $q$-th power Frobenius map. We will now describe the character theory of the finite group $U(\Fq)=U^F$ in terms of character sheaves.

\brk
The admissible pairs of type $(U,\pi^*\L_s)$ and their associated irreducible representations are rather straightforward to study. These precisely account for the $q$ 1-dimensional representations of $U(\Fq)$ that come from the quotient $U(\Fq)\onto \G_a(\Fq)=\Fq$. Hence let us restrict our attention to the admissible pairs of the type $(H,\L_s)$ where $s\neq 1$.
\erk

Let us begin by studying all the character sheaves associated with a minimal idempotent of the form $e_s:=e_{H,\L_s}$. Note that in this case the skew-symmetric isogeny $\phi_{\L_s}:\G_a=U/H\to \G_a^*=\G_a$ from Definition \ref{defadpair}(ii) is given by
\beq
y \mapsto sy^p-s^{1/p}y^{1/p}.
\eeq
Let $K_s$ be the kernel of the above isogeny and let $H\subset \t{K}_s\subset U$ be the subgroup such that $\tK_s/H=K_s$. Then by \cite[\S4.3]{De1} we have 
\beq
CS_{e_s}(U)=\{e_s\ast \delta_k|k\in K_s\},
\eeq
where $\delta_k$ is the delta sheaf supported at $(k,0)\in U$, i.e. the character sheaves associated with $e_s$ are its translates by $K_{s}$. Note that we have
\beq
K_s=\{k\in\k|s^pk^{p^2}=sk\}=\{k\in\k|sk^{p+1}\in \Fp\}.
\eeq
By results of \cite{De1}, the modular category $\M_{U,e_s}$ is the one associated with the metric group $(K_s,\theta_s)$ where the metric $\theta_s:K_s\to \Qlcl^\times$ is defined by $\theta_s(k)=\chi(sk^{p+1})$, where $\chi$ is our chosen non-trivial character of $\Fp$. (We refer to \cite[\S A]{Da} for details.)

Note that the idempotent $e_s$ is $F$-stable if and only if $s\in \Fq^\times$. Let us consider one such $s\in \Fq^\times$. Then we have the non-trivial character $\chi_s:H(\Fq)=\Fq\to \Qlcl^\times$ defined by $\chi_s(a)=\chi\circ \h{tr}_{\Fq|\Fp}(sa)$. In other words, $\chi_s$ is the ``trace of Frobenius'' function associated with the $F$-stable multiplicative local system $\L_s$. 

Let us now describe the $\f{L}$-packet 
\beq
\Irrep_{{e_s}}(U(\Fq))=\{\rho\in \Irrep(U(\Fq))|\rho(h) \mbox{ acts by the scalar $\chi_s(h)$ for each $h\in H(\Fq)$}\}
\eeq 
of irreducible characters associated with the $F$-stable minimal idempotent $e_s$ and its relationship with the set  
\beq
CS_{e_s}(U)^F=\{e_s\ast\delta_k|k\in K_s^F\}=\{e_s\ast\delta_k|k\in K_s, k^q=k\}
\eeq
of $F$-stable character sheaves in the associated $\f{L}$-packet. Consider the following skewsymmetric pairing
 defined using the commutator map and the character $\chi_s$ of $H(\Fq)$
\beq
(U/H)(\Fq) \times (U/H)(\Fq)\xto{[\cdot,\cdot]}H(\Fq)\xto{\chi_s}\Qlcl^\times,
\eeq
\beq
(x,y)\mapsto \chi_s(xy^p-x^py).
\eeq
By the definition of $K_s$ and the relationship between multiplicative local systems on $\G_a$ and characters of $\Fq$, we see that the kernel of the above skewsymmetric pairing is equal to $K_s(\Fq)=K_s^F$. The induced skewsymmetric pairing 
\beq
U(\Fq)/\t{K}_s(\Fq)\times U(\Fq)/\t{K}_s(\Fq) \rar{}\Qlcl^\times
\eeq 
is nondegenerate. Let $\t{K}_s(\Fq)\subset L\subset U(\Fq)$ be such that $L/\t{K_s}(\Fq)\subset U(\Fq)/\t{K}_s(\Fq)$ is a Lagrangian. Now all the irreducible representations in the $\f{L}$-packet associated with $e_s$ can be obtained in the following way:
\bit
\item Extend the character $\chi_s$ to a character $\t{\chi}_s:\t{K}_s(\Fq)\rar{}\Qlcl^\times$. The set of extensions is a $K_s^*$-torsor, where $K_s(\Fq)^*$ is the Pontryagin dual of $K_s(\Fq)$.
\item Extend $\t\chi_s$ further to a character $\t\chi'_s:L\to \Qlcl^\times$. This is possible since $L/\t{K}_s(\Fq)\subset U(\Fq)/\t{K}_s(\Fq)$ is isotropic. (The choice of this extension will eventually not matter.)
\item Form the induced representation $\rho:=\Ind_L^{U(\Fq)}\t\chi'_s$. Then we can check that $\rho\in \Irrep_{e_s}(U(\Fq))$ and that it only depends on the extension $\t\chi_s$ and not its further extension $\t\chi'_s$. It is easy to check that the character of $\rho$ is given by
\beq\label{e:esirrep}
\chi_\rho(y,b)=\begin{cases}
\frac{q^2}{|L|}\t\chi_s(y,b)=\frac{q^2}{|L|}\t\chi_s(y,0)\chi_s(b) & \text{if $y\in K_s(\Fq)$}\\
0 &\text{otherwise.}
\end{cases}
\eeq
\eit

Let $\omega:K_s(\Fq)\rar{}\Qlcl^\times$ be any character. Let us first suppose that $q$ is odd.  In this case it is easy to check that we can construct explicit extensions $\t{\chi}^\omega_s:\t{K}_s(\Fq)\rar{}\Qlcl^\times$ of $\chi_s$ parametrized by $\omega\in K_s(\Fq)^*$ and defined by $\t\chi^\omega_s(k,b):=\omega(k)\cdot\chi_s(b-\frac{k^{p+1}}{2})$ and that these give us all extensions. If $q$ is a power of $2$, then we choose one extension $\t\chi_s:\t{K}_s(\Fq)\rar{}\Qlcl^\times$ of $\chi_s$ and use it to obtain all extensions as $\t\chi^\omega_s:\t{K}_s(\Fq)\rar{}\Qlcl^\times$ defined by $\t\chi^\omega_s(k,b):=\omega(k)\cdot \t\chi_s(k,b)$. Hence in either case, we obtain a parametrization of the set $\Irrep_{e_s}(U(\Fq))$ by $K_s(\Fq)^*$. On the other hand, for the trace functions associated with the set $CS_{e_s}(U)^F$  we choose the functions $T_k:U(\Fq)\rar{}\Qlcl$ for $k\in K_s(\Fq)$ defined by 
\beq
T_k(y,b)=\delta_{ky}\cdot\t\chi_s(y,b).
\eeq With these choices (see also (\ref{e:esirrep})), the transition matrix between character sheaves and characters is the $K_s(\Fq)\times K_s(\Fq)^*$ matrix $S(e_s,F)$ with entries defined by
\beq
S(e_s,F)_{k,\omega}=\frac{q^2}{|L|}\omega(k).
\eeq
Also note that we have the metric group $(K_s,\theta_s)$ and an automorphism $F$ of the metric group defined by the Frobenius. This allows us to identify $K_s\cong K_s^*$. It is easy to check that we have $K_s(\Fq)^\perp=(K_s^F)^\perp=(1-F)K_s$. Hence the metric allows us to identify $K_s(\Fq)^*\cong K_s/(1-F)K_s$. This allows us to describe the matrices $S(e_s,F)$ using the $S$-matrix of the modular category $\M_{U,e_s}$ which is nothing but the nondegenerate bimultiplicative map $K_s\times K_s\rar{}\Qlcl^\times$ obtained using the metric $\theta_s$.

We now break our analysis into two cases:\\
(i) $q=p^{2r+1}$: Let $s\in \Fq^\times$. In this case $|K_s(\Fq)|=p$ and the irreducible representations in $\Irrep_{e_s}(U(\Fq))$ are $p^r$-dimensional.\\
(ii) $q=p^{2r}$: If $s\in \Fq^\times$ is such that $s^{\frac{q-1}{q+1}}=1$, then $|K_s(\Fq)|=p^2$ and the irreducible representations in $\Irrep_{e_s}(U(\Fq))$ are $p^{r-1}$-dimensional. In the case the transition matrix is equal (up to normalization) to the S-matrix of the modular category $\M_{U,e_s}$. If $s\in \Fq^\times$ is such that $s^{\frac{q-1}{q+1}}\neq 1$, then $|K_s(\Fq)|=1$ and the unique irreducible representation in $\Irrep_{e_s}(U(\Fq))$ is $p^{r}$-dimensional.

\section{Proof of Theorem \ref{ipf}}\label{pfofipf}
In this section we prove Theorem \ref{ipf} for an arbitrary algebraic group $G$. The proof is an application of the Grothendieck-Lefschetz trace formula. For notational convenience we use the integral symbol ($\int$) to denote pushforward with compact supports. If $\E\in \D(X)$, we use $\int\limits_{X}\E$ to denote $R\Gamma_c(X,\E)$.

The first equality $(T_{C,\psi},\chi_M)=<T_{C,\psi},\chi_M>$ is clear since $\bar{\chi_M(h,t)}=\chi_M(h^{-1},t)$ for each $(h,t)\in R$. Next we compute the stalk of the automorphism $\g_{C,\psi,M}$ at a point $t\in G$. The stalk
$$\g_{C,\psi,M}(t):(C\ast M)_t\to (C\ast M)_t$$ is given by
$$(C\ast M)_t=\int\limits_{h_1h_2=t} C_{h_1}\otimes M_{h_2} $$
$$\xto{\left(\psi(F^{-1}(h_2^{-1}h_1h_2))\circ\phi_C(h_2^{-1},h_1)\right)\otimes \phi_M(F^{-1}(h_2^{-1}h_1^{-1}h_2), h_2)}\int\limits_{h_1h_2=t} C_{F^{-1}(h_2^{-1}h_1h_2)}\otimes M_{F^{-1}(h_2^{-1}h_1^{-1}h_2)h_1h_2}$$

Consider the automorphism of the antidiagonal $\bar{\Delta}_t:=\{(h_1,h_2)\in G\times G|h_1h_2=t\}$ defined by 
$$(h_1,h_2)=(h_1,h_1^{-1}t)\mapsto ({F^{-1}(h_2^{-1}h_1h_2)}, {F^{-1}(h_2^{-1}h^{-1}_1h_2)}h_1h_2)=({F^{-1}(t^{-1}h_1t)}, {F^{-1}(t^{-1}h^{-1}_1t)}t).$$ The fixed point set of this automorphism is precisely the finite set $\bar{\Delta}_t\cap R$. The inverse of this automorphism corresponds to the Frobenius automorphism $ad(t)\circ F$ of $G$ under the identification $\bar{\Delta}_t\xto{p_1}G$. Hence by the Grothendieck-Lefschetz trace formula (trace of the induced map on cohomology equals the sum of traces of the stalks of the map over all fixed points, see also \cite[Lemma 4.4(iii)]{Bo2}) we deduce that 
\beq
\tr(\g_{C,\psi,M}(t))=\sum\limits_{\substack{(h_1,h_2)\in R\\h_1h_2=t}}T_{C,\psi}(h_1,h_2)\chi_M(h_1^{-1},h_2).
\eeq

Now $C\ast M$ is an object of $\D_G^F(G)\cong \bigoplus\limits_{\substack{<t> \in \\ H^1(F,G)}}D^b(\h{Rep}(G_0^t(\Fq)))$. Hence by definition (\ref{deftr}), 
\beq
\tr_F(\g_{C,\psi,M})=\sum\limits_{\substack{<t> \in \\ H^1(F,G)}}\frac{\tr(\g_{C,\psi,M}(t))}{|\Gtq|}
\eeq
\beq
=\sum\limits_{\substack{<t> \in \\ H^1(F,G)}}\sum\limits_{\substack{(h_1,h_2)\in R\\h_1h_2=t}}\frac{T_{C,\psi}(h_1,h_2)\chi_M(h_1^{-1},h_2)}{|\Gtq|}
\eeq
\beq
=\sum\limits_{\substack{<t> \in \\ H^1(F,G)}}\sum\limits_{\substack{(h,t)\in R}}\frac{T_{C,\psi}(h,h^{-1}t)\chi_M(h^{-1},h^{-1}t)}{|\Gtq|}
\eeq
\beq
=\sum\limits_{\substack{<t> \in \\ H^1(F,G)}}\sum\limits_{\substack{h\in \Gtq}}\frac{T_{C,\psi}(h,h^{-1}t)\bar{\chi_M(h,h^{-1}t)}}{|Z_{\Gtq}(h)|\cdot |\h{Conj}_{\Gtq}(h)|}
\eeq
\beq
=\sum\limits_{\substack{<(h,t)>\in G\backslash R}}\frac{T_{C,\psi}(h,h^{-1}t)\bar{\chi_M(h,h^{-1}t)}}{|Z_{\Gtq}(h)|}
\eeq
\beq
=\sum\limits_{\substack{<(h,t)>\in G\backslash R}}\frac{T_{C,\psi}(h,h^{-1}t)\bar{\chi_M(h,h^{-1}t)}}{|\Stab_G(h,t)|}
\eeq
\beq
=<T_{C,\psi},\chi_M>
\eeq
as desired. Here $\h{Conj}_{\Gtq}(h)$ denotes the conjugacy class of $h$ in $\Gtq$.

\section{Induction functors for braided crossed categories}

\subsection{Braided crossed categories associated with group extensions}\label{brcrgrex}

Let $G$ be a perfect quasi-algebraic group. Let $\Gamma$ be a discrete (possibly infinite) group and consider an extension
\beq
0\to G\to \t{G}\to \Gamma \to 0.
\eeq
%and $\g:G\to G$ an automorphism. Let us form the perfect group scheme $\t{G}:=G\rtimes \Z$ where the action of $\Z$ on $G$ is defined using $\g:G\to G$. Let us consider the coset $G\g\subset \t{G}$. Note that if we identify $G\g$ with $G$ (by $g\g\mapsto g$), then the conjugation action of $G$ on the coset $G\g$ corresponds to the $\g$-conjugation action of $G$ on itself.

Consider the category 
$$\D_G(\t{G}):=\bigoplus\limits_{G\tg\in \Gamma=\t{G}/G}\D_G(G\tg).$$ 
This is a $\Gamma$-graded monoidal category. For each $\g \in \Gamma$, let $\tg\in \t{G}$ be a lift. Consider the conjugation by $\tg^{-1}$ automorphism $\tg^{-1}:\t{G}\to \t{G}$ and define the functor $\tg:\DGt\to \DGt$ as pullback by the conjugation automorphism $\tg^{-1}$. This defines a monoidal action of $\Gamma$ on $\D_G(\t{G})$. 

We see that we have the following crossed braiding isomorphisms:
\blem
Let $\tg\in \t{G}$. Let $M\in \DGt$ and $N\in \D(G\tg)$. Then we have a natural isomorphism 
\beq
\beta_{N,M}: N\ast M \xto{\cong} \tg(M)\ast N \hbox{ in $\D(\t{G})$}.
\eeq
\elem

In fact we can define the structure of a braided $\Gamma$-crossed category on $\D_G(\t{G})$. It is also an $\mathfrak{r}$-category with the duality functor given by $\f{D}^-=\f{D}\circ \iota^*=\iota^*\circ \f{D}$, where $\f{D}:\D_G(\t{G})\to \D_G(\t{G})$ is the Verdier duality functor and $\iota:G\to G$ is the inversion map. Moreover it has a natural pivotal structure coming from the natural monoidal isomorphism $Id\cong (\f{D}^-)^2$ and the monoidal action of $\Gamma$ preserves this pivotal structure. Hence $\D_G(\t{G})$ is a braided $\Gamma$-crossed pivotal $\mathfrak{r}$-category.

\subsection{Induction functors}\label{secindfun}
In this section we define induction functors and study their properties in the setting of braided crossed categories which is slightly more general than the setting in \cite[\S4]{De3}. However the proofs from {\it loc. cit.} and \cite{BD2} can be readily adapted to our current set up. 

Consider an extension $0\to G\to\t{G}\to \Gamma\to 0$ as in \S\ref{brcrgrex}.  Let $\t{G}'\subset \t{G}$ be a subgroup that surjects onto $\Gamma$. Let $G':=G\cap \t{G}'$. Then we have the extension
\beq
0\to G'\to \t{G}'\to \Gamma \to 0.
\eeq

We can define the induction with compact supports functor
\beq
\ind_{G'}^G:\D_{G'}(\t{G'})\to \D_G(\t{G})
\eeq
as the composition $\D_{G'}(\t{G'})\to \D_{G'}(\t{G})\xto{\av_{G/{G'}}} \D_G(\t{G})$. By \cite[Prop. 3.4]{De3}, we have the structure of a weak semigroupal functor on $\ind_{G'}^G$. Moreover we can show that it is compatible with the $\Gamma$-actions and the braided $\Gamma$-crossed structure as well as the pivotal $\r$-category structure.

\bdefn
Let $G'\subset G$ be any closed subgroup in an algebraic group $G$. Let $e\in\DGp$ be any idempotent. We say that $e$ satisfies the Mackey condition with respect to $G$ if for every $g\in G-G'$ we have $e\ast\delta_g\ast e=0$, where $\delta_g$ is the delta sheaf supported at $g$.
\edefn

\brk
Note that the condition (iii) from Definition \ref{defadpair} of admissible pairs is equivalent to the Mackey condition above for the closed idempotent $e_{H,\L}\in \DGp$ associated with the pair $(H,\L)$.
\erk

Now let $e\in \D_{G'}(G')\subset \D_{G'}(\t{G}')$ be an idempotent satisfying the Mackey condition with respect to $G$ such that $\tg(e)\cong e$ for each $\tg\in \t{G}'$. 
\bprop\label{indmackey} (See \cite[Prop. 4.3]{De3}.)
In the situation above we have:\\
(i) The induced object $f:=\ind_{G'}^G e$ is an idempotent in $\D_G(G)\subset \D_G(\t{G})$ such that $\tg(f)\cong f$ for each $\tg\in \t{G}$.\\
(ii) If $M\in e\D_{G'}{(\t{G}')}$, then $\ind_{G'}^G M\in f\D_G(\t{G})$.\\
(iii) If $N\in \D(G)$, then $e\ast N_{G\backslash G'}\ast e=0$ and hence $e\ast N\ast e\cong e\ast N_{G'}\ast e$. Hence we can consider $e\ast N\ast e$ as an object of $\D(G')$.\\
(iv) Let $\tg\in \t{G}'$ and $N\in \D_G(G\tg)$. Then $e\ast N\cong e\ast N\ast e\cong e\ast N\ast \delta_{{\tg}^{-1}}\ast e\ast \delta_{\tg}$ is supported on $G'\tg\subset \t{G}'$. Hence for each $N\in \DGt$ we may consider $e\ast N\cong e\ast N|_{\t{G}'}$ as objects of $e\DGtp$.\\
(v) For each $N\in \DGt$, there is an isomorphism
\beq
f\ast N \stackrel{\cong}{\longrightarrow} \indg (e\ast N)
\eeq
functorial with respect to $N$.
\item[(vi)] For each $M\in e\DGtp$ we obtain functorial isomorphisms 
\beq
e\ast (\indg M) \stackrel{\cong}{\longrightarrow} e\ast M \hbox{ in } \D_{G'}(\t{G})
\eeq
\beq
e\ast (\Resg\circ\indg M) \stackrel{\cong}{\longrightarrow} e\ast M \hbox{ in } \DGtp,
\eeq
by taking the convolution of $e$ with the canonical morphism $\indg M\to M$ in $\D_{G'}(\t{G})$ and \linebreak
$\Resg\circ\indg M\to M$ in $\DGtp$. In particular we see that $e\ast f\cong e.$\\
(vii) The restriction \beq\label{indfun} \indg|_{e\DGtp}:e\DGtp\to f\DGt \eeq is strong semigroupal, compatible with the pivotal and braided $\Gamma$-crossed structures and induces a bijection on isomorphism classes of objects.\\
(viii) If the functor $M\mapsto e\ast M$ is isomorphic to the identity functor on $e\DGtp$, the functor (\ref{indfun}) is faithful.\\
(ix) If the functors $M\mapsto e\ast M$ and $N\mapsto f\ast N$ are isomorphic to the identity functors on $e\DGtp$ and $f\DGt$ respectively, the functor (\ref{indfun}) is an equivalence of braided $\Gamma$-crossed pivotal $\r$-categories, a quasi-inverse to which is provided by the functor (also see (iv) above) \beq f\DGt\ni N\mapsto e\ast N \in e\DGtp.\eeq

\eprop

\bpf
The proof is a straightforward generalization of the proof of \cite[Prop. 4.3]{De3}.
\epf

\subsection{Induction and sheaf-function correspondence}\label{indandsf}
Let $G$ be an algebraic group equipped with a Frobenius $F:G\to G$ which corresponds to the form $G_0$ over $\Fq$. As in \S\ref{semidir}, we form the semidirect product $\t{G}=G\rtimes\f{Z}$, where $1\in\f{Z}$ acts on $G$ by $F$. Suppose $G'\subset G$ is an $F$-stable subgroup or in other words $G'$ is obtained by base change from a closed subgroup $G'_0\subset G_0$ over $\Fq$. We have the subgroup $\t{G}':=G'\rtimes \f{Z}\subset \t{G}$ and by \S\ref{secindfun} we have the weak semigroupal functor
\beq
\indg:\D_{G'}(\t{G}')\to \D_G(\t{G}).
\eeq
In particular we get the functors
\beq\label{indgtg}
\indg:\D_{G'}({G}'F)\to \D_G({G}F)
\eeq
\beq
\indg:\D_{G'}({G}')\to \D_G({G})
\eeq
as well as
\beq
\indg:\D^{Weil}_{G'_0}({G}'_0)\to \D^{Weil}_{G_0}({G}_0).
\eeq

Next we interpret the functor (\ref{indgtg}) in terms of the usual notion of induction of representations. Note that the inclusion $G'\subset G$ induces a map $H^1(F,G')\to H^1(F,G)$. By \S\ref{astc} we have 
$$\D_{G'}(G'F)\cong \bigoplus\limits_{<t>\in H^1(F,G')}{D^b\Rep({G'}^t_0(\Fq))},$$
$$\D_G(GF)\cong \bigoplus\limits_{<t>\in H^1(F,G)}{D^b\Rep({G}^t_0(\Fq))}.$$

\blem\label{indisind}
Let $t\in G'$ and let $\O'\subset G'F$ be the $G'$-conjugacy class of $tF\in G'F$ and let $\O\subset GF$ be the $G$-conjugacy class of $tF$. Then the induction functor $\indg$ takes the full subcategory $D^b\Rep({G'}^t_0(\Fq))\cong\D_{G'}(\O')\subset \D_{G'}(G'F)$ to the full subcategory $D^b\Rep({G'}^t_0(\Fq))\cong\D_G{(\O)}\subset \D_G(GF)$ and the restriction of the functor $\indg$ to $\D_{G'}(\O)$ can be identified with the functor 
$$\ind_{{G'}^t_0(\Fq)}^{\Gtq}:D^b\Rep({G'}^t_0(\Fq))\to D^b\Rep({G}^t_0(\Fq)).$$
\elem
\bpf
This follows from the equivalences $\D_{G'}(\O')\cong \D_{{G'}_0^t(\Fq)}(tF)$ and  $\D_{G}(\O)\cong \D_{{G}_0^t(\Fq)}(tF)$ and the definition of induction functors using the averaging functors.  
\epf

In this situation we also have an induction map between the function spaces
\beq
\ind_{G'_0}^{G_0}:\Fun([G'_0](\Fq))\to\FunG
\eeq
obtained from induction of class functions 
$$\ind_{G'^t_0(\Fq)}^{G^t_0(\Fq)}:\Fun(G'^t_0(\Fq))^{G'^t_0(\Fq)}\to\Fun(G^t_0(\Fq))^{G^t_0(\Fq)}.$$

Using Lemma \ref{indisind} we get
\blem\label{inddgf}
Let $M\in \D_{G'}(G'F)$. Then we have $\ind_{G'_0}^{G_0}(\chi_M)=\chi_{\indg(M)}.$
\elem

We recall the following from \cite{Bo2}:
\bprop\label{sfandind} (\cite[Prop. 4.12]{Bo2})
Let $(C,\psi)\in \D_{G'_0}^{Weil}(G'_0)$. Then $T_{\indg(C,\psi)}=\ind_{G_0'}^{G_0}(T_{C,\psi})$, i.e. induction of conjugation equivariant Weil sheaves is compatible with the sheaf-function correspondence (provided we take all pure inner forms into account).
\eprop

\section{Reduction process for neutrally unipotent groups}\label{rpfornug}
In this section we assume that $G$ is a neutrally unipotent group equipped with an $\Fq$-structure given by a Frobenius $F:G\to G.$ In this section we reduce the proofs of Theorems \ref{main2} and \ref{main3} to the case of Heisenberg idempotents. In \S\ref{heiscase} we will prove these theorems in the Heisenberg case.

Let $f\in \widehat{G}^F$ and let $f_0\in \DGn$ be the corresponding geometrically minimal idempotent. Now without loss of generality (by passing to a pure inner form of $G_0$ if necessary) by Theorem \ref{Fstabe}, we may assume that $f_0$ comes from an admissible pair $(H_0,\L_0)$ for $G_0$. Let $G_0'$ be the normalizer of the admissible pair. Let $e_0\in \D_{G'_0}(G'_0)$ be the corresponding Heisenberg idempotent. By extension of scalars we get an $F$-stable admissible pair $(H,\L)$ for $G$ with normalizer $G'$. {Note that here $F$ may be different from the original one we started with since we may have to choose a different pure inner form in order to obtain the admissible pair $(H_0,\L_0)$.} Let $e\in \DGp$ be the corresponding Heisenberg idempotent on $G'\subset G.$ The idempotent $e\in \DGp$ satisfies the Mackey condition with respect to $G$.   As before, let $\t{G}=G\rtimes \f{Z}$ and $\t{G}'=G'\rtimes \f{Z}$. Now the subgroup $\t{G}'\subset \t{G}$ and $e\in \D_{G'}(\t{G}')$ satisfy the conditions of Proposition \ref{indmackey}. Moreover since $e\in \DGp$ and $f\in \DG$ are closed idempotents, we conclude using Proposition \ref{indmackey}({ix}) that we have an equivalence 
\beq
\indg: e\D_{G'}(\t{G}')\cong f\D_G(\t{G})
\eeq
of braided $\f{Z}$-crossed pivotal $\r$-categories. Moreover by \cite[Thm. 1.52]{BD2} the induction functor  $\indg$ preserves perverse sheaves in $e\D_{G'}(\t{G}')$ up to a shift by $\dim(G/G')$. In particular this induces an equivalence of modular categories 
\beq\label{mgf}
\indg:\M_{G',e}\cong \M_{G,f}
\eeq
as well as their respective module categories 
\beq\label{mgff}
\indg:\M_{G'F,e}\cong \M_{GF,f}
\eeq and identifies their natural module traces. It is clear that we also have an equivalence of their respective module categories 
\beq\label{mgff+}
\indg:\M_{G'F,e}^+\cong \M_{GF,f}^+
\eeq equipped with the positive traces.

\subsection{Reduction of Theorem \ref{main2} to the Heisenberg case}\label{redofmain2}
We continue to use all the notation from \S\ref{rpfornug}. We will now prove Theorem \ref{main2} for the general minimal idempotent $f\in\DG$ assuming that it holds for the Heisenberg idempotent $e\in \DGp$.

Theorem \ref{main2}(i) for $f$ follows from (\ref{mgf}) and (\ref{mgff}) assuming the corresponding statement for $e$. To prove the remaining statements, we first prove two auxiliary lemmas.

\blem (See \cite[Lem. 6.24]{Bo2})\label{span}
Let $f_0\in \DGn$ be any idempotent. Then:\\
(i) The ``trace of Frobenius'' functions associated with the objects of $\DGn$ span the space $\FunG$.\\
(ii) The ``trace of Frobenius'' functions associated with the objects of $f_0\DGn$ span the space $T_{f_0}\FunG$.
\elem
\bpf
The statement (i) is precisely Lemma 6.24 in \cite{Bo2}. Now if $C_0\in \DGn$, then $T_{f_0\ast C_0}=T_{f_0}\ast T_{C_0}$ by Lemma \ref{eqcompwithconv}. Hence (ii) follows from (i).
\epf

\blem\label{indisom}
The induction of class functions on the pure inner forms induces an isomorphism of Frobenius algebras
$$\ind_{G'_0}^{G_0}: T_{e_0}\Fun([G'_0](\Fq))\xto\cong T_{f_0}\FunG.$$
\elem
\bpf
It is clear that the induction map $\ind_{G'_0}^{G_0}: \Fun([G'_0](\Fq))\xto\cong \FunG$ commutes with the linear functionals $\lambda':\Fun([G'_0](\Fq))\to \Qlcl$ and $\lambda:\Fun([G_0](\Fq))\to \Qlcl$ defined by (\ref{linfun}).
By Proposition \ref{sfandind}, $\ind_{G'_0}^{G_0}(T_{e_0})=T_{f_0}$. We have an equivalence of braided categories $$\indg:e\DGp\xto{\cong} f\DG$$ compatible with the action of $F$ on both sides and hence an equivalence 
$$\indg:e_0\D^{Weil}_{G'_0}(G'_0)\xto\cong f_0\D^{Weil}_{G_0}(G_0).$$
The lemma then follows from Proposition \ref{sfandind} and Lemma \ref{span}.

\epf

Now by (\ref{mgff}) $\indg$ induces a bijection between the sets of simples $\O_{\M_{G'F,e}}$ and $\O_{\M_{GF,f}}$. Now we use the assumption that Theorem \ref{main2}(ii) and  (iii) hold for $e\in \DGp$. With this assumption we have the bijection $\O_{\M_{G'F,e}}\cong \Irrep_{e_0}(G'_0)$ and hence $\{\chi_{M'}\}_{M'\in \O_{\M_{G'F,e}}}$ is a basis of $T_{e_0}\Fun([G'_0](\Fq))$ (see Remark \ref{simplebasis}). For $M'\in \O_{\M_{G'F,e}}$, let $M=\indg(M')\in \O_{\M_{GF,f}}$. Then by Lemmas \ref{inddgf} and \ref{indisom} we see that $\chi_M\in T_{f_0}\FunG$ and $\{\chi_M\}_{M\in \O_{\M_{GF,f}}}$ is a basis of $T_{f_0}\FunG$. Hence the irreducible representation (of some pure inner form $G_0^t(\Fq)$) corresponding to $M$ indeed lies in $\Irrep_{f_0}(G_0)$ and this sets up a bijection between $\O_{\M_{GF,f}}$ and $\Irrep_{f_0}(G_0)$ proving Theorem \ref{main2}(ii) and (iii) for the idempotent $f\in \DG$.

\subsection{Reduction of Theorem \ref{main3} and Conjecture \ref{conjmain} to the Heisenberg case}\label{redofmain3}
Statement (ii) of Theorem \ref{main3} is an immediate consequence of (i) using Corollary \ref{cor1}. We now prove the statement (i) for $f$ assuming it for $e$.

Note that for an endomorphism $a$ in $\D_{G'}(G'F)$ we have $\tr_{G,F}(\indg(a))=\tr_{G',F}(a)$ for the traces on $\D_{G'}(G'F)$ and $\D_G(GF)$ defined in \S\ref{traceondgf}. Also by \cite[Thm. 1.52]{BD2}, $\dim G-d_f=\dim G'-d_e$ and $2d_f\equiv 2d_e \mod 2$. Moreover by (\ref{mgf}) $\dim({\M_{G',e}})=\dim(\M_{G,f})$ and by (\ref{mgff}) the traces $\tr_F,e$ and $\tr_{F,f}$ are compatible with the induction functor and so are the positive traces $\tr_{F,e}^+$ and $\tr_{F,f}^+$. Hence we see that Theorem \ref{main3} holds for $f$ if it holds for $e$.

This argument also proves that Conjecture \ref{conjmain} for $f\in \DG$ is equivalent to the same statement for $e\in \DGp$. Hence the conjecture is reduced to the case of Heisenberg idempotents.

\section{The Heisenberg case}\label{heiscase}
In this section we complete the proofs of Theorems \ref{main2} and \ref{main3} in the case of Heisenberg idempotents.

We begin by recalling the setup in the Heisenberg case. As before we assume that $G$ is a neutrally unipotent group equipped with an $\Fq$-structure given by a Frobenius $F:G\to G$. We assume that we have a Heisenberg admissible pair $(H_0,\L_0)$ for the group $G_0$ over $\Fq$. Let $e_0\in \DGn$ be the corresponding geometrically minimal (Heisenberg) idempotent and $e\in \DG$ the Heisenberg idempotent on $G$ obtained by extension of scalars from $e_0$. Note that in this case $n_e=\dim H$. Let $U=G^\circ$. Then $U$ is an $F$-stable connected unipotent group and let $U_0$ be the corresponding unipotent group defined over $\Fq$. Let $\Gamma:=\Pi_0(G)=G/U$.

By definition $U/H$ is commutative and the homomorphism $\phi_\L:U/H\to (U/H)^*$ (from condition (ii) of Definition  \ref{defadpair}) is an isogeny. The finite group $\Gamma$ acts on $U/H$ and $\phi_\L$ is $\Gamma$-equivariant. Let $K_{\L}:=\ker(\phi_\L)$.  Let $\theta:K_\L\to \Qlcl^\times$ be the corresponding quadratic form defined in \cite[\S A.10]{Bo1}. By \cite{De1} we have $e\D_U(U)\cong D^b\M(K_\L,\theta)$, where $\M(K_\L,\theta)$ is the pointed modular category defined by the metric group $(K_\L,\theta)$.

Now by definition the pair $(H,\L)$ is normalized by $G$ and is also $F$-stable. Note that we have the induced automorphism $F:\Gamma\to \Gamma$. Using the Frobenius action, we form the semidirect products $\t{\Gamma}:=\Gamma\rtimes \f{Z}$ and $\t{G}:=G\rtimes \f{Z}$. Thus we have an extension
\beq
0\to U\to \t{G}\to \t{\Gamma}\to 0
\eeq
of the discrete group $\t{\Gamma}$ by $U$. The admissible pair $(H,\L)$ is normalized the group $\t{G}$. By \S\ref{brcrgrex} the category $\D_U(\t{G})$ is a braided $\t{\Gamma}$-crossed pivotal $\r$-category and so is the Hecke subcategory $e\D_U(\t{G})$.

Let $\tM_{\t{G},e}\subset e\D_U(\t{G})$ denote the full subcategory of those objects whose underlying complexes are perverse sheaves shifted by $\dim H$. For $\t{g}\in \t{G}$ we let $\tM_{U\t{g}, e}\subset e\D_U(U\t{g})$ be defined similarly. Then we have
\beq\label{brtgcr}
\tMeg=\bigoplus\limits_{U\tg \in \t{G}/U}{\tMtg}.
\eeq

The following result is proved in \cite{De1}:
\bthm\label{De1main}
(i) For each $U\tg\in \t{G}/U$, $\tMtg$ is a semisimple abelian category with finitely many simple objects and we have $D^b\tMtg\cong e\D_U(\t{G})$.\\
(ii) The category $\tMeg\subset e\D_U(\tG)$ is closed under convolution. It is a braided $\t{\Gamma}$-crossed spherical \emph{fusion} category with trivial component $\tM_{U,e}=\M_{U,e}\cong \M(K_\L,\theta)$. 
\ethm

The category $e\D_G(\tG)$ is a braided $\f{Z}$-crossed pivotal $\r$-category. We have $e\D_G(\tG)\cong e\D_U(\tG)^\Gamma$.
Let $\M_{\tG,e}\subset e\D_G(\tG)$ be the full subcategory of objects whose underlying complex is a perverse sheaf shifted by $\dim H$. Similarly for an integer $n$ we define $\M_{GF^n,e}\subset e\D_G(GF^n)\subset e\D_G(\t{G})$. We have
\beq\label{zgrading}
\Mtge=\bigoplus\limits_{n\in \f{Z}}\M_{GF^n,e}.
\eeq

Now using Theorem \ref{De1main} we obtain:
\blem\label{equivariantization}
The category $\Mtge\subset e\D_G(\tG)$ is closed under convolution. In fact $\Mtge$ is a braided $\f{Z}$-crossed spherical fusion category with trivial component $\Meg$ and can be obtained as a $\Gamma$-equivariantization:
\beq
\Mtge\cong \left({{{\tM_{\tG,e}}}}\right)^{{\Gamma}}.
\eeq
We have $D^b\Mtge \cong e\D_G(\tG)$.
\elem

\subsection{Proof of Theorem \ref{main2}}\label{pfofmain2}
We now prove Theorem \ref{main2} in the Heisenberg case.

Using (\ref{zgrading}), Lemma \ref{equivariantization} along with \cite[\S6]{ENO} we see that each $\M_{GF^n,e}$ (and in particular $\M_{GF,e}$) is an invertible $\Meg$-module category equipped with an $\Meg$-module trace obtained from the natural spherical structure on the braided $\f{Z}$-crossed category $\Mtge$. This trace is normalized in the sense of (\ref{normtr}) by Remark \ref{crbrandmod}.  This completes the proof of Theorem  \ref{main2}(i).

We make the following general observation:
\blem\label{lemdfgtoweil}
Let $H$ be any  connected algebraic group equipped with a Frobenius $F:H\to H$. We have a functor \beq
\alpha:\D^F_H(H)\to \D^{Weil}(H_0)
\eeq 
such that for each $M\in \D^F_H(H)$ the underlying $\Qlcl$-complexes of $M$ and $\alpha(M)$ are the same and we have $\chi_M=t_{\alpha(M)}$ i.e. the character of $M$ is equal to the ``trace of Frobenius'' function associated with $\alpha(M)$.
\elem
\bpf
If $(M,\phi_M)\in\D^F_H(H)$ where $M\in \D(H)$ and $\phi_M$ is the equivariance structure of $M$, then we have isomorphisms of stalks
\beq\label{dfgtoweil}\phi_M(h,F(h)):M_{F(h)}\xto\cong M_{h} \hbox{ for each $h\in H$}.\eeq 
In other words we obtain a Weil structure on $M\in \D(H)$. To be precise, let $\Delta_{F}\subset H\times H$ denote the graph of the Frobenius $F:H\to H$. Then we set $\alpha(M,\phi_M)=(M,\phi_M|_{\Delta_{F}})\in \D^{Weil}(H_0)$. 

Now suppose that $h\in H_0(\Fq)$. Then by (\ref{chim}) we have $\chi_{M,\phi_M}(h)=tr(\phi_M(h,1))$ where 
\beq\label{chimu}\phi_M(h,1):M_1\xto{\cong} M_1.\eeq
On the other hand by (\ref{dfgtoweil}) for the ``trace of Frobenius function'' associated with $\alpha(M)$ we have $t_{\alpha(M,\phi_M)}=tr(\phi_M(h,h))$. Now since $H$ is connected, there exists $v\in H$ such that $h=vF(v)^{-1}$ by Lang's theorem. Then we have a commutative diagram
$$\xymatrixcolsep{7pc}\xymatrix{
M_1\ar[r]^{\phi_M(h,1)}\ar[d]_{\phi_M(v,1)} & M_1\ar[d]^{\phi_M(v,1)}\\
M_h\ar[r]^{\phi_M(h,h)} & M_h.\\
}$$
Hence $tr({\phi_M(h,1)})=tr({\phi_M(h,h)})$. Hence we see that $\chi_M=t_{\alpha(M)}$ for each $M\in \D^F_H(H)$. 
\epf

%We have
%\beq
%\D_G(GF)\cong \D_U(GF)^\Gamma\cong\left(\bigoplus\limits_{Ug\in G/U}\D_U(UgF)\right)^\Gamma\cong\left(\bigoplus\limits_{Ug\in G/U}D^b\Rep(U_0^g(\Fq))\right)^\Gamma,
%\eeq
%where $U_0^g$ is the unipotent group over $\Fq$ defined by the Frobenius $ad(g)\circ F:U\to U$ for $g\in G$. We note that since $U$ is connected, the conjugation action of $U$ on $UgF$ is transitive, or equivalently $U^g_0$ has no nontrivial pure inner forms  (using Lang's theorem).

%We now characterize  the full subcategory of $D^b\Rep(U^g_0(\Fq))$ corresponding to the full subcategory $e\D_U(UgF) \subset \D_U(UgF)$. 

%Since the $F$-stable Heisenberg admissible pair $(H,\L)$ is normalized by $G$, the pair $(H,\L)$ is $ad(g)\circ F$-stable and gives us a Heisenberg admissible pair $(H_0^g,\L^g_0)$ for each of the groups $U^g_0$ and $G^g_0$ and a geometrically minimal idempotent $e_0^g\in \D_{U_0^g}(U^g_0)$. 
Let us now return to our original setting where we have the admissible pair $(H_0,\L_0)$ for the neutrally unipotent group $G_0$ and the corresponding geometrically minimal idempotent $e_0\in \DGn$. Taking the functions associated with sheaves, we get the character
\beq
t_{\L_0}:H_0(\Fq)\to \Qlcl^\times
\eeq
and the corresponding idempotent function $t_{e_0}=\frac{1}{q^{\dim H}}\cdot t_{\L_0}\in \Fun(G_0(\Fq))^{\Fun(G_0(\Fq))}$ which is supported on the normal subgroup $H_0(\Fq)\normal G_0(\Fq)$.

Let $\Rep_{e_0}(H_0(\Fq))\subset \Rep(H_0(\Fq))$ denote the full subcategory of representations of $H_0(\Fq)$ where the group acts by the character $t_{\L_0}$. Note that we have $D^b\Rep(H_0(\Fq))\cong \D^F_H(H)$ the equivariant derived category for the $F$-conjugation action of $H$ on itself. Since $H$ is connected, this action is transitive.% or equivalently of representations in which the idempotent $t_{e^g_0}$ acts as the identity or equivalently of representations whose character lies in the subspace $t_{e_0^g}\Fun(U^g_0(\Fq))$. We now prove:

\blem\label{repe0}
Let $W\in \Rep(H_0(\Fq))$. Let $W_{loc}$ be the corresponding local system in $\D_H^F(H)$. Then $W\in \Rep_{e_0}(H_0(\Fq))$ if and only if $W_{loc}\in e\D_H^F(H)$.
\elem
\bpf
Let $W$ be the 1-dimensional representation of $H_0(\Fq)$ corresponding to the character $t_{\L_0}$, i.e. $W$ is the unique simple object in the category $\Rep_{e_0}(H_0(\Fq))$ (which is equivalent to $\Vec$ as an abelian category). It is clear that $\D_H^F(H)\ni W_{loc}\cong \L$ (see also Lemma \ref{lemdfgtoweil}). Since of $e=\L\otimes \can_H$ and $\L$ is a multiplicative local system, it is clear that $e\ast \L\cong \L$, i.e. $\L\cong W_{loc}\in e\D^F_H(H)$. 

Now since $H$ is connected and unipotent, $\D_H^F(H)\subset \D(H)$ is a full subcategory and hence \linebreak $e\D_H^F(H)\subset e\D(H)$ is a full subcategory. But $e\D(H)\cong D^b\Vec$ and all objects of $e\D(H)$ are of the form $V\otimes\L$ where $V\in \Vec$ (thought of as a constant local system on $H$). Since $\L\in e\D^F_H(H)$ we see that $e\D^F_H(H)=e\D(H)$. The lemma now follows.
\epf

To prove Theorem  \ref{main2}(ii) and (iii)  we have to compare the set of simple objects in the category $\M_{GF,e}\subset e\D_G(GF)$ and the set $\Irrep_{e_0}(G_0)$. Let $\O\subset G$ be an orbit of the $F$-conjugation action. Without any loss of generality, for convenience of notation (after possibly modifying the Frobenius $F$ by an inner automorphism of $G$) we assume that $1\in \O$. Hence we have $\D^F_G(\O)\cong D^b\Rep(G_0(\Fq))$. 

Let $\Rep_{e_0}(G_0(\Fq))\subset \Rep(G_0(\Fq))$ denote the full subcategory of representations of $G_0(\Fq)$ in which the normal subgroup $H_0(\Fq)$ acts by the character $t_{\L_0}$. %Then $\Irrep_{e_0}(G_0(\Fq))$ is the set of simple objects of $\Rep_{e_0}(G_0(\Fq))$.

We now prove:
\bprop\label{edfgo}
Let $W\in \Rep(G_0(\Fq))$. Let $W_{loc}$ be the corresponding local system in $\D_G^F(\O)\subset \D^F_G(G)$. Then $W\in \Rep_{e_0}(G_0(\Fq))$ if and only if $W_{loc}\in e\D_G^F(G)$. We have an equivalence $e\D^F_G(\O)\cong D^b\Rep_{e_0}(G_0(\Fq))$.
\eprop
\bpf
Let $Q$ denote the algebraic group $G/H$. Then we have the induced Frobenius $F:Q\to Q$ and the algebraic group $Q_0$ over $\Fq$. Since $H$ is connected, we have $Q_0(\Fq)=G_0(\Fq)/H_0(\Fq)$. Let us define the group $Q_H:=\{g\in G|gF(g)^{-1}\in H\}$. This is the subgroup of $G$ that takes $H\subset \O$ to itself under the $F$-conjugation action  of $G$ on $\O$. We have that $H=Q_H^\circ$ and $Q_H/H\cong Q_0(\Fq)$. We have $G_0(\Fq)\subset Q_H$. We have an equivalence $\D^F_G(\O)\cong \D^F_{Q_H}(H)$, where the latter is the category of $Q_H$-equivariant sheaves on $H$ for the $F$-conjugation action of $Q_H$ on $H$. Hence we have $\D_G^F(\O)\cong \D_H^F(H)^{Q_0(\Fq)}$. The proposition now follows from Lemma \ref{repe0}.
\epf

Finally to prove Theorem \ref{main2}(ii) and (iii) we apply Proposition \ref{edfgo} to each $F$-conjugacy class $\O\in H^1(F,G)$.

\subsection{Proof of Theorem \ref{main3}}\label{pfofmain3}
It suffices to prove Theorem \ref{main3}(i) since (ii) is an immediate consequence.

Let $W\in \Rep_{e^g_0}(G_0^g(\Fq))$ for some $g\in G$ and let $M_W:=W_{loc}[\dim G+\dim H]\in {\M_{GF,e}}$ be the corresponding  object obtained by Theorem \ref{main2}. We have 
\beq
\tr_{F,e}^+(\id_{M_W})=\FPdim(M_W)\hbox{ and }\tr_F(\id_{M_W})=(-1)^{\dim G+\dim H}\cdot\frac{\dim(W)}{|G_0(\Fq)|}.
\eeq Hence to prove Theorem \ref{main3}(i) it suffices to prove that 
\beq\label{suff1}
\FPdim(M_W)=\frac{q^{\dim G}\cdot\sqrt{\dim \Meg}}{q^{d_e}}\cdot \frac{\dim(W)}{|G_0(\Fq)|}
\eeq
for each such $W$. As before for ease of notation, without loss of generality we may assume $W\in  \Rep_{e_0}(G_0(\Fq))$ corresponding to the trivial inner form. As before, let $\O\subset G$ denote the $F$-conjugacy class of $1\in G$. Then by our assumption, $M_W\in \M_{GF,e}\subset \D_G(GF)$ is supported on $\O F$. 

Recall that we have defined $U=G^\circ$ and $\Gamma=G/U$. Also recall that by (\ref{brtgcr}) and Theorem \ref{De1main} we have the braided $\tGam$-crossed spherical fusion category $\tM_{\tG,e}$. Set 
\beq
\tM_{G,e}:=\bigoplus\limits_{Ug\in G/U\subset \tGam}\tM_{Ug,e},
\eeq
\beq
\tM_{GF,e}:=\bigoplus\limits_{Ug\in G/U\subset \tGam}\tM_{UgF,e}.
\eeq
Then $\tM_{GF,e}$ is an $\tM_{G,e}$ bimodule category and by Lemma \ref{equivariantization}, we have
\beq\label{eq1}
\Meg=(\tM_{G,e})^\Gamma,
\eeq
\beq\label{eq2}
\M_{GF,e}=(\tM_{GF,e})^\Gamma.
\eeq

$\tM_{G,e}$ is a (faithfully graded) braided $\Gamma$-crossed spherical fusion category with trivial component $\tM_{U,e}=\M_{U,e}$. Hence $\dim(\tM_{G,e})=|\Gamma|\cdot \dim(\M_{U,e})$ and hence after taking the $\Gamma$-equivariantization we conclude that (see also \cite[Prop. 2.17]{De2}) 
\beq\label{catdim}
\dim(\Meg)=|\Gamma|^2\cdot \dim(\M_{U,e}).
\eeq 

The object $M_W\in \M_{GF,e}$ is supported on $\O F$. Let $\O'\subset \Gamma$ be the $F$-conjugacy class of $1\in \Gamma.$ By Lang's theorem
\beq
\O=\coprod\limits_{Ug\in \O'\subset G/U}Ug.
\eeq
Using (\ref{eq2}) we think of $M_W$ as an object of $(\tM_{GF,e})^\Gamma$. Let $M'\in \tM_{GF,e}$ denote the underlying object. Then $M'\in \bigoplus\limits_{Ug\in \O'\subset G/U}\tM_{UgF,e}.$ Let $M'_{UF}\in \tM_{UF,e}$ be the projection of $M'$ in $\tM_{UF,e}$. Then we have 
\beq\label{FPdim}
\FPdim(M_W)=|\O'|\cdot\FPdim(M'_{UF})=\frac{|\Gamma|}{|\Gamma_0(\Fq)|}\cdot \FPdim(M'_{UF}).
\eeq
The object $M'_{UF}\in \tM_{UF,e}\subset e\D_U(UF)$ corresponds (in the sense of Theorem \ref{main3} applied to $U$) to the restriction $W'=\Res^{G_0(\Fq)}_{U_0(\Fq)}(W)$. {We have $\Gamma_0(\Fq)=G_0(\Fq)/U_0(\Fq)$ and $|U_0(\Fq)|=q^{\dim U}=q^{\dim G}$ since $U$ is a connected and unipotent group.} Hence by (\ref{catdim}) and (\ref{FPdim}), to prove \ref{suff1}, we are reduced to showing that
\beq\label{suff2}
\FPdim(M'_{W'})=\frac{\sqrt{\dim \M_{U,e}}}{q^{d_e}}\cdot {\dim(W')}
\eeq
for each $W'\in \Rep_{e_0}(U_0(\Fq))$, where $M'_{W'}\in \tM_{UF,e}\subset e\D_U(UF)$ is the object corresponding to $W'$ according to Theorem \ref{main2}. Now, $\tM_{U,e}=\M_{U,e}$ is a pointed modular category. Hence all simple objects in the invertible $\tM_{U,e}$-module category $\tM_{UF,e}$ have equal Frobenius-Perron dimension.  Hence
\beq\label{dim1}
|\Irrep_{e_0}(U_0(\Fq))|.\FPdim(M'_{W'})^2=\dim(\M_{U,e})
\eeq
Also $(U_0(\Fq)/H_0(\Fq))$ is commutative, so all irreducible representations in $\Rep_{e_0}(U_0(\Fq))$ have the same dimension and we have
\beq\label{dim2}
|\Irrep_{e_0}(U_0(\Fq))|.\dim(W')^2=|U_0(\Fq)/H_0(\Fq)|=\frac{q^{\dim G}}{q^{\dim H}}=q^{2d_e}.
\eeq
The equality  (\ref{suff2}) follows from equations (\ref{dim1}) and (\ref{dim2}). This completes the proof of Theorem \ref{main3}.

\section{Grothendieck ring of Weil sheaves}\label{grows}
In this section we study a certain Grothendieck rings of Weil sheaves. We will prove that the Grothendieck ring of $F$-stable character sheaves on a neutrally unipotent group $G$ is isomorphic to the algebra $\FunG$ (see Corollary \ref{t:main}).

\subsection{Categorical constructions}\label{catcon}
We begin with some generalities for abstract modular categories equipped with a modular autoequivalence (or equivalently an action of $\f{Z}$). Let $(\C,\otimes,\un)$ be a $\Qlcl$-linear modular category equipped with an action of $\f{Z}$ by modular autoequivalences. Let $a:\C\to \C$ be the modular autoequivalence corresponding to $1\in \f{Z}$. We form the $\f{Z}$-equivariantization of $\C$ and denote it by $(\C^{\f{Z}},\otimes,\un)$. The objects of $\C^\bZ$ can be thought of as pairs $(C,\psi)$, where $C\in \C$ and $\psi:a^{-1}(C)\xto{\cong} C$. This is a (nonsemisimple) spherical braided monoidal abelian category. Hence the Grothendieck ring $K_0(\C^{\f{Z}})$ is a commutative ring. 

Let us consider the trivial modular category $\Vec$ equipped with the trivial action of $\f{Z}$. Then $\Vec^{\f{Z}}$ is the category of finite dimensional $\Qlcl$-vector spaces equipped with an automorphism $\psi$, or equivalently an action of $\f{Z}$. We have the ring homomorphism $K_0(\Vec^\bZ)\rar{}\Qlcl$ that takes the class of $(V\in\Vec,\psi\in\Aut V)\in \Vec^\bZ$ to $tr(\psi)\in \Qlcl$. We can identify $K_0(\Vec^\bZ)$ with the group ring $\bZ[\Qlcl^\times]$ so that the element $[\alpha]\in\bZ[\Qlcl^\times]$ corresponding to group element $\alpha\in \Qlcl^\times$ gets identified with the class of a $1$-dimensional vector space equipped with the automorphism of multiplication by $\alpha$. Our morphism $K_0(\Vec^\bZ)\rar{}\Qlcl$ takes $[\alpha]\in \bZ[\Qlcl^\times]$ to $\alpha\in \Qlcl$.

For any modular category $\C$ equipped with a $\f{Z}$ action, we have the braided functor $\Vec^{\f{Z}}\to \C^{\f{Z}}$ which induces a ring homomorphism $K_0(\Vec^\bZ)\rar{}K_0(\C^\bZ)$. Now define the commutative $\Qlcl$-algebra 
\beq
A_{\C,a}:=K_0(\C^\bZ)\otimes_{K_0(\Vec^\bZ)}\Qlcl.
\eeq 

\bdefn\label{d:linfun}
The algebra $A_{\C,a}$ is equipped with a canonical linear functional $\lambda_{\C,a}:A_{\C,a}\rar{}\Qlcl$ such that $\lambda_{\C,a}(1)=1$ defined as follows:\\
The functor $\C^\bZ\rar{}\Vec^\bZ$ that takes $C\in\C^\bZ$ to $\Hom(\mathbf{1},C)\in\Vec^\bZ$ induces a morphism of $K_0(\Vec^\bZ)$-modules $K_0(\C^\bZ)\rar{}K_0(\Vec^\bZ)$ and therefore a $\Qlcl$-linear functional $\lambda_{\C,a}:A_{\C,a}\rar{}\Qlcl$.
\edefn

\begin{prop}
(i) We have $\dim A_{\C,a}=|\O_\C^\bZ|$, where $\O_\C^\bZ$ is the set of $\bZ$-invariant (equivalently $a$-invariant) elements of ${\O_\C}$.\\
(ii) The bilinear form $(a_1,a_2)\longmapsto\lambda_{\C,a}(a_1a_2)$ for $a_1,a_2\in A_{\C,a}$, is nondegenerate. In other words, $A_{\C,a}$ is a Frobenius algebra. 

\end{prop}
\begin{proof}
Let us describe $K_0(\cM^\bZ)$ as a module over $K_0(\Vect^\bZ)=\bZ[\Qlcl^\times]$. We have
\[
\cM=\bigoplus_O \cM_O, \quad \cM^\bZ=\bigoplus_O \cM_O^\bZ, \quad K_0(\cM^\bZ)=\bigoplus_O K_0(\cM_O^\bZ),
\]
where $O$ runs through the set of $\bZ$-orbits in the set ${\O_\cM}$ of simple objects of $\C$ and $\cM_O\subset\cM$ is the full subcategory consisting of objects of $\cM$, all of whose irreducible components are in $O$. It is easy to see that if $|O|=n$, then $K_0(\cM_O^\bZ)\subset K_0(\cM^\bZ)$ is a $\bZ[\Qlcl^\times]$-submodule, which is noncanonically isomorphic to $\bZ[\Qlcl^\times]/I_n$, where $I_n\subset\bZ[\Qlcl^\times]$ is the ideal generated by elements of the form $[\zeta]-1$, where $\zeta\in \Qlcl^\times$ is such that $\zeta^n=1$. If $n>1$, then $\zeta^n=1$ for some $\zeta\in\Qlcl\setminus\{1\}$, so $K_0(\cM_O^\bZ)\tensor_{\bZ[\Qlcl^\times]}\Qlcl=0$. If $n=1$ i.e. if $O=\{u\}$ for some $u\in \O_\C^\bZ$, then $\dim_{\Qlcl} \left( K_0(\cM_{\{u\}}^\bZ)\tensor_{\bZ[\Qlcl^\times]}\Qlcl\right)=1$. So $A_{\C,a}$ is the direct sum of the lines $L_u:=K_0(\cM^\bZ_{\{u\}})\tensor_{\bZ[\Qlcl^\times]}\Qlcl$, where $u\in\O_\C^\bZ$. This implies statement (i).

Given $u,u'\in \O_{\cM}^\bZ$, the pairing $L_u\times L_{u'}\rar{}\Qlcl$, induced by the bilinear form $(a_1,a_2)\longmapsto\lambda_{\C,a}(a_1a_2)$, is nonzero if and only if $u'$ is dual to $u$. This implies (ii).
\end{proof}

\subsection{The setup of neutrally unipotent groups} 
We now return to our setting of a neutrally unipotent group $G$ equipped with an $\Fq$-structure defined by a Frobenius $F:G\to G$. Recall from \S\ref{frobal} that we have the Frobenius algebra $\FunG$ of class functions on all the pure inner inner forms $G^g_0(\Fq)$ under convolution. 

Now let $e\in \hat{G}^F$ be an $F$-stable minimal idempotent and $e_0\in \DGn$ the corresponding geometrically minimal idempotent. Then we have the modular category $\M_{G,e}$ equipped with the modular autoequivalence $F=(F^{-1})^*:\Meg \to \Meg$. Then using the categorical construction from \S\ref{catcon} we obtain the Frobenius algebra $A_{e_0}:=A_{\Meg,F}=K_0(\M_{G,e}^\f{Z})\otimes_{K_0(\Vec^\Z)}\Qlcl$ equipped with the linear functional $\lambda_{e_0}:=\lambda_{\Meg,F}:A_{e_0}\to \Qlcl$. 

The sheaf-function correspondence (for conjugation-equivariant Weil sheaves) yields an algebra morphism 
\beq
T:A_{e_0}\to T_{e_0}\FunG
\eeq
 by Theorem \ref{Bo2main}. We have the linear functional $\l:\FunG\to \Qlcl$ defined by (\ref{linfun}).

\begin{lem}\label{l:key}
In the setting above, the following diagram commutes:
\[
\xymatrix{
  A_{e_0} \ar[rrr]^-{T} \ar[drr]_{\l_{e_0}} & & & T_{e_0}\FunG \ar[ld]^{q^{{(\dim G + n_e)}} \cdot\l} \\
  & & \Qlcl &
   }
\]
where $n_e\in\bZ$ is as in Theorem \ref{BDmain}.
\end{lem}

The proof of the lemma is given in \S\ref{ss:proof-l:key} below. Let us discuss some corollaries.

\begin{cor}\label{c:1}
The map $T: A_{e_0}\to T_{e_0}\FunG$ is an algebra isomorphism. In particular the Frobenius algebra $A_{e_0}$ is semisimple.
\end{cor}

\begin{proof}
The injectivity follows from the lemma using the nondegeneracy of the form $(a,b)\longmapsto\lambda_{e_0}(ab)$, where $a,b\in A_{e_0}$. The surjectivity follows from Lemma \ref{span}(ii).
\end{proof}

As a consequence, we have:
\begin{cor}\label{t:main}
The map $T:\prod\limits_{e\in \widehat{G}^F} A_{e_0} \rar{} \FunG$ defined using the sheaf-function correspondence is an isomorphism.
\end{cor}

\brk
The algebra $A_{e_0}$ has a natural basis (well defined up to rescaling) parametrized by the set of $F$-stable character sheaves in the $\f{L}$-packet $CS_e(G)$ and the algebra $\prod\limits_{e\in \widehat{G}^F} A_{e_0}$ has a basis parametrized by the set $CS(G)^F$. On the other hand, the algebra $\FunG$ has a basis formed by its set of minimal idempotents (which is in a 1-1 correspondence with the set $\Irrep(G_0)$ of irreducible representations of all the pure inner forms). The relationship between these two bases is described by the crossed $S$-matrices studied in this paper.
\erk
Let us now derive some results about dimensions of irreducible representations (cf. Thm. \ref{t:corthm}).
\bcor\label{c:sqdim}
Let $W\in \Irrep_{e_0}(G^t_0(\Fq))\subset \Irrep_{e_0}(G_0)$. Then $e_W:=\frac{\dim(W)}{|\Gtq|}\cdot \chi_W$ is a minimal idempotent in $T_{e_0}\FunG$. Let $\epsilon_W$ be the corresponding minimal idempotent in $A_{e_0}$.\\ 
(i) Then $\l_{e_0}(\epsilon_W)=q^{\dim G+n_e}\cdot\frac{\dim(W)^2}{|\Gtq|^2}$. \\
(ii) Hence, we obtain 
\beq
\sum\limits_{W\in \Irrep_{e_0}(G_0)}\frac{\dim(W)^2}{|\Gtq|^2}=\frac{1}{q^{\dim G+n_e}}.
\eeq 
(iii) Furthermore, if $G$ is connected then we have 
\beq
\sum\limits_{W\in \Irrep_{e_0}(G_0(\Fq))}\dim(W)^2=q^{2d_e}.
\eeq
\ecor 
\bpf
Statement (i) follows directly from the definition (\ref{linfun}) of $\l$ and Lemma \ref{l:key}. To prove (ii), we observe that the set of all minimal idempotents in $A_{e_0}$ is equal to $\{\epsilon_W|W\in \Irrep_{e_0}(G_0)\}$. Hence we must have $\sum\limits_{W\in \Irrep_{e_0}(G_0)}{\epsilon_W}=1$ and by definition $\l_{e_0}(1)=1$. Statement (iii) immediately follows from (ii). 
\epf

\subsection{Proof of Lemma \ref{l:key}}\label{ss:proof-l:key}
The following is an immediate consequence of \cite[Prop. 4.18]{Bo2}:
\blem\label{l:Bo2}
Talking the stalk at $1\in G$ (i.e. $C\mapsto C_1$) defines functors 
\beq\DG\to \D_G(1),\eeq
\beq\D_{G_0}^{Weil}(G_0)\to \D_{G_0}^{Weil}(1)\eeq
and taking $G$-invariants (i.e. $V\mapsto V^G$) defines functors
\beq D^b\Rep(G)\cong D^b\Rep(\Pi_0(G))\cong \D_G(1)\to D^b\Vec,\eeq
\beq\D_{G_0}^{Weil}(1)\to (D^b\Vec)^{\f{Z}}.\eeq
Let $C\in \D_{G_0}^{Weil}(G_0)$ and let $T_C\in \FunG$ denote the associated trace of Frobenius function. Then the number $q^{\dim G}\cdot\l(T_C)$ is equal to the trace of the Frobenius automorphism associated with the object $(C_1)^G\in (D^b\Vec)^{\f{Z}}$.
\elem

Using Definition \ref{d:linfun} we obtain:
\blem\label{l:aux}
Let $C\in \Meg^{\f{Z}}\subset \D_{G_0}^{Weil}(G_0)$. Then $\l_{e_0}(C)$ is equal to the trace of the Frobenius automorphism associated with the object $\Hom_{\Meg}(e,C)\in \Vec^{\f{Z}}$. 
\elem

To complete the proof of Lemma \ref{l:key} we will prove that  $(C_1)^G[-2n_e](-n_e)$ and $\Hom_{\Meg}(e,C)$ are isomorphic as objects of $(D^b\Vec)^{\f{Z}}$ for each $C\in \Meg^\Z$.

Recall that we have the dualizing functor $\f{D}^-:\DG\to\DG$ such that we have natural identifications
\beq
\RHom_{\DG}(C\ast D,\delta_1)\cong \RHom_{\DG}(C,\f{D}^-D) \mbox{ for $C,D\in \DG$.}
\eeq 
Also note that for $C\in \DG$ we have the adjunction
\beq
\RHom_{\DG}(C,\delta_1)\cong \RHom_{\D_G(1)}(C_1,\Qlcl).
\eeq

Now let $C$ be any object of $\Meg^\Z$. Using the observations above along with the semisimplicity of the category $\Meg$ and Theorem \ref{BDmain} we obtain a sequence of isomorphisms (as objects of $(D^b\Vec)^{\f{Z}}$):
\beq\Hom_{\Meg}(e,C)^*\cong\Hom_{\Meg}(C,e)\eeq
\beq\cong\RHom_{D^b\Meg}(C,e)\eeq
\beq\cong\RHom_{\DG}(C,e)\eeq
\beq\label{e:dual}\cong\RHom_{\DG}(C,\f{D}^-e[2n_e](n_e))\eeq
\beq\cong\RHom_{\DG}(C,\f{D}^-e)[2n_e](n_e)\eeq
\beq\cong\RHom_{\DG}(C\ast e,\delta_1)[2n_e](n_e)\eeq
\beq\cong\RHom_{\DG}(C,\delta_1)[2n_e](n_e)\eeq
\beq\cong\RHom_{\D_G(1)}(C_1,\Qlcl)[2n_e](n_e)\eeq
\beq\cong\RHom_{D^b\Rep(G)}(C_1,\Qlcl)[2n_e](n_e)\eeq
\beq\cong\RHom_{\Vec}((C_1)^G,\Qlcl)[2n_e](n_e)\eeq
\beq\cong\left((C_1)^G\right)^*[2n_e](n_e).\eeq

Taking duals we conclude that  $C_1^G[-2n_e](-n_e)\cong \Hom_{\Meg}(e,C)$ in $(D^b\Vec)^{\f{Z}}$. Now taking the associated traces of Frobenius, and using Lemmas \ref{l:Bo2} and \ref{l:aux} we conclude that $q^{n_e}\cdot q^{\dim G}\cdot \l(T_C)=\l_{e_0}(C)$. The factor $q^{n_e}$ here appears due to the Tate twist $(-n_e)$. The completes the proof of Lemma \ref{l:key}.

\end{document}